\documentclass{article}
\usepackage{amsmath, amsthm, amssymb}
\usepackage[mathscr]{eucal}
\usepackage{mathrsfs}
\usepackage{psfrag}
\usepackage{color}
\usepackage{eucal}
\usepackage[dvips]{graphicx}
\usepackage{multirow}
\usepackage{fancyhdr}
\usepackage{enumerate}

\newtheorem{proposition}{Proposition}[section]
\newtheorem{theorem}[proposition]{Theorem}
\newtheorem{corollary}[proposition]{Corollary}
\newtheorem{lemma}[proposition]{Lemma}

\newtheorem{definition}[proposition]{Definition}
\newtheorem{remark}[proposition]{Remark}

\newtheorem{condition}[proposition]{Condition}

\newtheorem{rem}[proposition]{Remark}

\newcommand{\xbr}[1]{\left\{#1\right\}}

\newcommand{\x}{X^{ \epsilon,\delta} }
\newcommand{\eps}{\epsilon}
\newcommand{\R}{\mathbb{R}}

\newcommand{\Span}{\mathop{\rm span}}
\newcommand{\one}{\mathbf{1}}
\newcommand{\Pp}{\mathbf{P}}
\newcommand{ \parbar}[1]{ { \left( #1 \right)} }
\newcommand{\ONE}{{\bf 1}}
\newcommand{\PpConv}{\stackrel{\Pp}{\longrightarrow}}

\numberwithin{equation}{section}
\numberwithin{proposition}{section}

\begin{document}

\title{Scaling Limits and Exit Law for Multiscale Diffusions}

\author{Sergio Angel Almada \\
  \multicolumn{1}{p{.7\textwidth}}{\centering\emph{Department of Statistics and Operations Research, University of North Carolina,
 304 Hanes Hall CB \#3260, Chapel Hill, NC 27599, email: salmada3@unc.edu}}\\ \\
  Konstantinos Spiliopoulos\footnote{Corresponding author} \\
  \multicolumn{1}{p{.7\textwidth}}{\centering\emph{Department of Mathematics \& Statistics, Boston University,
 111 Cummington Street, Boston MA 02215,\\  email: kspiliop@math.bu.edu}}}


%

\date{\today}

\maketitle

\begin{abstract}
In this paper we study the fluctuations from the limiting behavior of small noise random perturbations of diffusions with multiple scales. The result is then applied to the exit problem for multiscale diffusions,  deriving the limiting law of the joint distribution of the exit time and exit location. We  apply our results to the first order Langevin equation in a rough potential, studying both fluctuations around the typical behavior and  the conditional limiting exit law, conditional on the rare event of going against the underlying deterministic flow.
\end{abstract}
\section{Introduction}

Let $T>0$ be given and consider a small random perturbation of dynamical system by a Wiener process. In particular, consider the $d$-dimensional process $X^{\epsilon}=\{X_{t}^{\epsilon},0\leq t\leq T\}$ satisfying the stochastic
differential equation (SDE)
\begin{equation}
dX_{t}^{\epsilon}=b^{\epsilon}\left(  X_{t}^{\epsilon}\right)  dt+\sqrt{\epsilon}\sigma^{\epsilon}\left(X_{t}^{\epsilon}\right)
dW_{t},\hspace{0.2cm}X_{0}^{\epsilon}=x_{0}, \label{Eq:DynamicalSystemMotivation}%
\end{equation}
where  $\epsilon\downarrow0$ and $W_{t}$ is a standard $d$-dimensional Wiener process. The functions
$b^{\epsilon}(x),\sigma^{\epsilon}(x)$ are assumed to be sufficiently smooth.

If $b^{\epsilon}(x)\rightarrow b(x)$ and $\sigma^{\epsilon}(x)\rightarrow\sigma(x)$ as $\epsilon\downarrow 0$, where $b(x)$ and  $\sigma(x)$ are nice functions, then asymptotic behavior such as law of large numbers, central limit theorems and large deviations have been extensively studied in the literature,e.g., \cite{Freidlin1978,FWBook} and the references therein. Scaling limits of (\ref{Eq:DynamicalSystemMotivation}) under the effect of different perturbations of the dynamics and of the initial condition are also studied in the recent article \cite{SergioBakhtin2011}.

In this article, we assume that the functions $b^{\epsilon}(x)$ and $\sigma^{\epsilon}(x)$ are fast oscillating, in particular we set $b^{\epsilon}(x)=\frac{\epsilon}{\delta}b\left(x,\frac{x}{\delta}\right)+c\left(x,\frac{x}{\delta}\right)$ and $\sigma^{\epsilon}(x)=\sigma\left(x,\frac{x}{\delta}\right)$, where $\delta=\delta(\epsilon)\downarrow0$ as $\epsilon\downarrow0$. The functions $b(x,y),c(x,y)$ and $\sigma(x,y)$ are assumed to be smooth and periodic with period $\rho$ in every direction with respect to the second variable. Homogenization of such equations has been studied extensively in the literature, see for example \cite{BLP,PS}. Large deviations were studied in~\cite{DupuisSpiliopoulos, FS} and related importance sampling schemes were developed in \cite{DupuisSpiliopoulosWang,DupuisSpiliopoulosWang2}. Moreover, special cases of this general equation (e.g., with $b(x,y)=-\nabla Q(y), c(x,y)=-\nabla V(x)$ and $\sigma(x,y)=\textrm{constant}$) have been suggested as models for studying rough energy landscapes that describe certain proteins and  their folding and binding properties.  A representative, but by no means complete, list of references is~\cite{DupuisSpiliopoulosWang2,Janke,Zwanzig}.

Our goal in this paper is twofold. First, we study scaling limits  under different perturbations of the drift and of the initial condition. We are interested in fluctuations around the typical behavior of $X^\eps _{t}$ as $\epsilon,\delta\downarrow 0$ when both the initial condition and the drift follow a scaling limit in finite time. It turns out that depending on the scaling and on the order that $\epsilon$ and $\delta$ go to zero, we have different limiting behavior. The result is presented in Theorem~\ref{T:CLT2}. It is interesting to note that, in contrast to the case without fast oscillations, in the case considered here, additional drift terms may appear in the equation that the fluctuation process satisfies, see Remark \ref{R:AdditionalTerms}. At this point we mention the articles
~\cite{DupuisSpiliopoulos,Freidlin1978,FS,Guillin,KlebanerLipster} for some related moderate and large deviations results, even though the fluctuations analysis done in the current paper is not covered, as far as the authors know, by the existing literature. The analysis of these scaling limits is summarized in Theorem \ref{T:CLT2} and allows us then to study the exit distribution in the limit as $\epsilon,\delta\downarrow 0$ (see Theorem \ref{thm: Main}) in the case in which the typical behavior of $X^\eps$ exits the domain transversally in finite time.

Another byproduct of this analysis is the study of the effect that perturbations by small but fast oscillations of  small noise  dynamical system have on exit time for such diffusions conditioned on rare events, see Theorem \ref{T:ConditionalExitTime}. We investigate this question in the case of the first order Langevin equation for both a periodic and for a random rough potential, see Remark \ref{R:RandomEnvironment}. It turns out that the limiting law of the exit time conditioned on the event of going against the deterministic flow, appropriately normalized, follows Gaussian distribution with enhanced variance (as compared to the small noise not oscillating case) due to the fast oscillations, see Remark \ref{R:ConclusionsRoughPotential}.

The rest of the paper is organized as follows. In Section \ref{S:Notation} we establish notation and mention examples and preliminary results that will be used throughout.
Section \ref{S:CLT} contains the corresponding central limit theorem, whereas Section~\ref{S:Exit} contains the analysis of the joint limiting law for the exit time and exit point. In Section~\ref{S:ConditionalExitLaw} we apply the results of Sections~\ref{S:CLT} and~\ref{S:Exit} to the first order Langevin equation in a rough environment. In particular, we state the
related central limit theorem and study the conditional exit law of a one dimensional small noise diffusion process in a rough environment in the limit as the fluctuations and noise intensity go to zero.

\section{The set-up}\label{S:Notation}

Let $T>0$ be given and consider the $d$-dimensional process $X^{\epsilon}%
\doteq\{X_{t}^{\epsilon},0\leq t\leq T\}$ satisfying the stochastic
differential equation (SDE)
\begin{equation}
dX_{t}^{\epsilon}=\left[  \frac{\epsilon}{\delta}b\left(  X_{t}^{\epsilon
},\frac{X_{t}^{\epsilon}}{\delta}\right)  +c\left(  X_{t}^{\epsilon}%
,\frac{X_{t}^{\epsilon}}{\delta}\right)  +\epsilon^{a_{1}/2}\Psi^\eps \left(  X_{t}^{\epsilon}%
,\frac{X_{t}^{\epsilon}}{\delta}\right)\right]  dt+\sqrt{\epsilon}%
\sigma\left(  X_{t}^{\epsilon},\frac{X_{t}^{\epsilon}}{\delta}\right)
dW_{t}, \label{Eq:LDPandA1}%
\end{equation}
with initial condition given by $X_{0}^{\epsilon}=x_{0}+\epsilon^{a_{2}/2}\xi^{\epsilon}$. Here $\xi^\eps$ is a family of random variables that converges in distribution to $\xi^0$ as $\eps \to 0$, $\delta=\delta(\epsilon)\to 0$ as $\epsilon\to 0$ and $W_{t}$ is a standard $d$-dimensional Wiener process. Also, we assume that the functions $b , c, \Psi^\eps$ and $\sigma$ satisfy the following conditions:
\begin{condition}
\label{A:Assumption1}
\begin{enumerate}
\item The functions $b(x,y),c(x,y),\sigma(x,y)$, and $\Psi^\eps(x,y)$ are, for each $\eps >0$, periodic with period $\rho$ in the second variable, $C^{1}(\mathcal{Y})$ in $y$ and $C^{2}(\mathbb{R}^{d})$ in $x$ with all partial derivatives continuous and globally bounded in both variables. Here $\mathcal{Y}=\mathbb{T}^{d}$ denotes the $d$-dimensional torus.
\item As $\eps \to 0$ $\Psi^\eps \to \Psi$ uniformly in each variable and $\Psi$ satisfies the same regularity conditions as any $\Psi^\eps$.
\item The diffusion matrix $\sigma\sigma^{T}$ is uniformly nondegenerate.
\end{enumerate}
\end{condition}

We are interested in the  following cases of interaction
\begin{equation}
\lim_{\epsilon\downarrow0}\frac{\epsilon}{\delta}=%
\begin{cases}
\infty & \text{Regime 1,}\\
\gamma\in(0,\infty) & \text{Regime 2,}
\end{cases}
\label{Def:ThreePossibleRegimes}%
\end{equation}
Here $\gamma$ is taken to be $\gamma = \infty$ in Regime 1.

We borrow some notation from \cite{DupuisSpiliopoulos}, where the large deviations principle for SDE~\eqref{Eq:LDPandA1} was established, in order to present our results.

\begin{definition}
\label{Def:ThreePossibleOperators} For each one of the Regimes $i=1,2$ defined in (\ref{Def:ThreePossibleRegimes}), and $x\in\mathbb{R}^{d}$, define the operators
\begin{align*}
\mathcal{L}_{x}^{1}  &  =b(x,\cdot)\cdot\nabla_{y} +\frac{1}{2}\textrm{tr}\left[\sigma
(x,\cdot)\sigma(x,\cdot)^{T}\nabla^{2}_{y}\right], \text{ and }\\
\mathcal{L}_{x}^{2}  &  =\left[  \gamma b(x,\cdot)+c(x,\cdot)\right]
\cdot\nabla_{y}+\gamma\frac{1}{2}\textrm{tr}\left[\sigma(x,\cdot)\sigma(x,\cdot)^{T}\nabla^{2}_{y}\right].
\end{align*}
For each $x \in \R^d$, the domain of $\mathcal{L}_{x}^{i}$ is given by $\mathcal{D}(\mathcal{L}_{x}^{i})=\mathcal{C}%
^{2}(\mathcal{Y})$, for $i=1,2$.
\end{definition}

Note that the existence of a unique smooth invariant measure for the operator $\mathcal{L}_{x}^{i}$, $i=1,2$, is immediately implied by Condition \ref{A:Assumption1} (see Theorem 3.3.4 and Section 3.6.1 of \cite{BLP}). We impose the following condition for the invariant measure in Regime 1:
\begin{condition}
\label{A:Assumption2} For each $x \in \R^d$, let $\mu^{i}(dy|x)$ be the unique invariant measure corresponding to the operator $\mathcal{L}_{x}^{i}$ equipped with periodic boundary conditions in $y$.

Under Regime 1, we assume the standard centering condition (see \cite{BLP}) for the drift term $b$:
\[
\int_{\mathcal{Y}}b(x,y)\mu^{1}(dy|x)=0.
\]
The variable $x$ is being treated as a parameter here.
\end{condition}

We note that under Conditions~\ref{A:Assumption1} and~\ref{A:Assumption2}, for each $l\in\{1,\ldots,d\}$, there is a unique
twice differentiable function $\chi_{\ell}(x,y)$ that is $\rho-$ periodic in every direction in $y$, that solves
the following cell problem (for a proof see \cite{BLP}, Theorem 3.3.4):
\begin{equation}
\mathcal{L}_{x}^{1}\chi_{l}(x,y)=-b_{l}(x,y),\quad\int_{\mathcal{Y}}%
\chi_{l}(x,y)\mu^{1}(dy|x)=0, \quad l=1,...,d. \label{Eq:CellProblem}%
\end{equation}
We write $\chi=(\chi_{1},\ldots,\chi_{d})$. With this in hand, it will become useful to define a function $\lambda_{i}(x,y)$, $i=1,2$, as follows:
\begin{definition}
\label{Def:ThreePossibleFunctions} For each one of the Regimes $i=1,2$ defined in (\ref{Def:ThreePossibleRegimes}), let $\lambda_{i}:\mathbb{R}%
^{d}\times\mathcal{Y}\rightarrow\mathbb{R}^{d}$ be given by
\begin{align*}
\lambda_{1}(x,y)  &  =\left(  I+\nabla_y\chi(x,y)\right)
 c(x,y), \text{ and }\\
\lambda_{2} (x,y)  &  =\gamma b(x,y)+c(x,y),
\end{align*}
where $\chi=(\chi_{1},\ldots,\chi_{d})$ is defined by (\ref{Eq:CellProblem}) and $I$ is the identity matrix.

Moreover, let $\bar \lambda_i : \R^d \to \R^d$ be given by
\begin{equation*}
 \bar{\lambda}_{i}(x)=\int_{\mathcal{Y}}\lambda_{i}(x,y)\mu^{i}(dy|x),
\end{equation*}
and let $\bar X^i_s (x)$ be the flow generated by $\bar{\lambda}_{i}$. That is, for each $x \in \R^d$, $\bar X ^i_s (x)$ is the solution to the ordinary differential equation
\[
\bar{X}^{i}_{t} (x)=x+\int_{0}^{t}  \bar{\lambda}_{i}(\bar{X}^{i}_{s})ds.
\]
\end{definition}

We remark here that under Condition \ref{A:Assumption1}, the invariant measure  $\mu^{i}(dy|x)$ is certainly $C^{1}$ in the $x-$variable (see Section 3.6.1 in \cite{BLP}) and consequently $\bar{\lambda}_{i}(\cdot)$ is $C^{1}$.
Hence, the ODE for $\bar X^i$ is well defined and has unique solution in each regime. Moreover, Theorem 2.8 in \cite{DupuisSpiliopoulos} guarantees weak convergence of $X^{\epsilon}_{\cdot}$ to $\bar{X}^{i}_{\cdot}$ in $\mathcal{C}([0,T])$ for any $T>0$. Further, it is easy to observe that in our case Theorem 2.8 in \cite{DupuisSpiliopoulos} implies that for all  $\eta>0$ and $i=1,2$, we have
\begin{equation}
\lim_{\epsilon \to 0}\mathbb{P}\left\{ \sup_{0\leq t\leq T}\left|X^{\epsilon}_{t}-\bar{X}^{i}_{t} (x_0) \right|>\eta\right\}=0
, \quad T>0.\label{Eq:LLN}
\end{equation}

Our first objective is to understand the limit of the fluctuations process
\[
\eta^{\epsilon}_{t}=\frac{X^{\epsilon}_{t}-\bar{X}^{i}_{t}}{\beta^{\epsilon}},\quad\textrm{ as }\epsilon\downarrow 0,
\]
where $\beta^{\epsilon}$ is the appropriate normalization rate. Our second objective is to prove a limit theorem for an exit problem of $X^\eps$ using the limiting result for the fluctuations process. That is, for a smooth $C^2$-hypersurface $M$ in $\R^d$, we are interested in studying the joint distribution of the hitting time
\begin{equation*}
\tau ^{\epsilon }=\inf \left \{t\ge 0:X^{\epsilon}(t)\in M \right \},
\end{equation*}
and the exit location $X_\eps(\tau^\eps)\in M$ as $\eps\to 0$ under the assumption that $\tau_\eps<\infty$ with probability $1$. Precise assumptions on the joint geometry of the vector field $\bar{\lambda}$ and the surface $M$ will be given in Section \ref{S:Exit}.

We conclude this section with a remark for the degenerate case $\epsilon/\delta\rightarrow 0$.
\begin{remark}
In the case $\epsilon/\delta\rightarrow 0$, the results of \cite{DupuisSpiliopoulos} indicate that the correct pair $(\mathcal{L}_{x},\lambda(x,y))$ is that of Regime $2$ with $\gamma=0$, as long as there is a unique invariant measure to the corresponding first order operator. Due to the fact that this operator is first order, the existence and uniqueness of an invariant measure is a difficult issue and requires additional assumptions on the vector field $c(x,y)$. For this reason and for the additional technical difficulties in treating the related Poisson equation (\ref{Eq:CellProblemCLT}), we decided not to treat this case in the current paper. See however, Corollary \ref{C:CLT_Regime3} for the case $\gamma=0$, in dimension $d=1$ when $c(x,y)>0$.
\end{remark}

\section{Analysis of fluctuations}\label{S:CLT}

In this section we establish a limit theorem for the correction of $X^\eps - \bar X^i$ in each case. Before stating our results in this direction, we need additional notation.

Let us consider the auxiliary PDE problem
\begin{equation}
\mathcal{L}_{x}^{i}\Xi_{i}(x,y)=-\left(\lambda_{i}  \left( x,y\right) - \bar{\lambda}_{i}(x)\right),\quad\int_{\mathcal{Y}}%
\Xi_{i}(x,y)\mu^{i}(dy|x)=0,\hspace{0.1cm} \label{Eq:CellProblemCLT}
\end{equation}
for $i=1,2$. Since, by definition, the right hand side of the PDE averages to zero with respect to the corresponding invariant measure $\mu^{i}(dy|x)$, Fredholm alternative implies that the function $\Xi_{i}(x,y)$ is uniquely defined, $\rho-$periodic in $y$, twice differentiable in both variables and with bounded derivatives (see Theorem 3.3.4 in \cite{BLP}). The function
$\Xi_{i}$ will be used to understand the dependence of terms like
\begin{equation}
I^{\epsilon, i}_{t}=\int_{0}^{t}\left(\lambda_{i}\left(X^{\epsilon}_{s},\frac{X^{\epsilon}_{s}}{\delta}\right)-\bar{\lambda}_{i}(X^{\epsilon}_{s})\right)ds\label{Eq:FluctuationsTerm}
\end{equation}
on $\epsilon$ and $\delta$.

Let us give some preliminary notation. For a function $f : \R^d \times \mathcal{Y} \to \R^d$, denote $\bar f_i:\R^d \to \R^d$ the average with respect to $\mu^i$:
\[
\bar f_{i}(x) = \int_{ \mathcal{Y} } f (x,y) \mu^i (dy|x).
\]
Also, define the following functions,
\begin{align*}
\Psi_{1} (x,y)&=\left(I+\nabla_y \chi(x,y)\right)\Psi(x,y),\nonumber\\
J_{1} (x,y)&=c\nabla_{y} \Xi_{1}(x,y), \\
q_{1}(x,y)&=(I+\nabla_y \chi%
)(x,y)\sigma(x,y)\sigma^{T}(x,y)(I+\nabla_y \chi%
)^{T}(x,y), \nonumber
\end{align*}
and
\begin{align*}
\Psi_{2} (x,y)&=(I+\nabla_y \Xi_{2}(x,y)) \Psi(x,y),\nonumber\\
J_{2}( x,y)&=\left(b(	I+\nabla_{y}\Xi_{2})+\frac{1}{2}\textrm{tr}\left[\sigma \sigma^T \nabla_{y}\nabla_{y}\Xi_{2}\right]\right)(x,y), \\
q_{2}(x,y)&= (I+\nabla_y \Xi_{2})(x,y)\sigma(x,y)\sigma^{T}(x,y)(I+\nabla_y \Xi_{2})^{T}(x,y). \nonumber
\end{align*}
Further, for $x \in \R^d$, let $\Phi^i_{x}$  be the linearization of $\bar X^{i}$ along the orbit of $x$:
\begin{equation}
 \frac{d}{dt}\Phi^i_{x}(t)=D\bar{\lambda}^{i}( \bar X^i_t )\Phi^i_{x}(t), \text{  } \Phi^{i}_{x}(0)=x
\end{equation}
where $D\bar{\lambda}^{i}$ is the Jacobian matrix of $\bar{\lambda}^{i}$.
We are  now ready to state our results.

\begin{theorem} \label{T:CLT2}
Let $T>0$, and assume Conditions \ref{A:Assumption1}-\ref{A:Assumption2}. Set $\theta_{1}^{\epsilon}=\frac{\delta}{\epsilon}$,
$\theta_{2}^{\epsilon}=\frac{\epsilon}{\delta}-\gamma$, $m=\min\left\{\frac{1}{2},\frac{\alpha_{1}}{2},\frac{\alpha_{2}}{2}\right\}$,
 $$ \ell_{i} = \lim_{ \eps \to 0 } \frac{\epsilon^{m}}{\theta^{\epsilon}_{i}} \in [0,\infty],$$
and
\begin{equation*}
\beta^{\epsilon}_{i} (\ell_{i} )=
\begin{cases}
\theta_{i}^{\epsilon} & ,\ell_{i}=0,\\
\epsilon^{m} & , \ell_{i} \in (0,\infty]
\end{cases}.
\end{equation*}
Let $\bar{\eta}^i$ be a process of the Ornstein-Uhlenbeck type such that
\begin{eqnarray}
 d\bar{\eta}_{t}^i&=&D\bar{\lambda}_{i}(\bar{X}^{i}_{t}(x_{0}))\bar{\eta}_{t}dt+\left[ \ell^{-1}_{i} \one (\ell_{i} \in(0,\infty] ) + \one (\ell_{i} =0 ) \right ]\bar{J}_{i}(\bar{X}^{i}_{t}(x_{0}))dt+\nonumber\\
  & &\quad +\one \left(\ell_{i}\neq 0) \right)\left[\one \left(m=a_{1}/2\right)\bar{\Psi}_{i}^i(\bar{X}^{i}_{t}(x_{0}))dt+\one \left(m=1/2\right)\bar{q}^{1/2}_{i}(\bar{X}^{i}_{t}(x_{0}))dW_{t}\right]\nonumber\\
\bar{\eta}_{0}&=&\xi_{0}\one \left(m=a_{2}/2 \textrm{ and } \ell_{i}\neq 0\right).\nonumber
\end{eqnarray}
Then, for each $\eps > 0$, there is a process $\eta^\eps ( \ell_{i})$, such that
\[
X^{\epsilon}_t =  \bar X^{i}_{t} + \beta^\eps (\ell_{i}) \eta^{\eps}_t (\ell_{i})
\]
holds with probability $1$ for every $t>0$, and $\eta^\eps ( \ell_{i}) \to \bar{\eta}^i (\ell_{i})$, as $\eps \to 0$, in distribution in $\mathcal{C}\left([0,T];\mathbb{R}^{d}\right)$.
\end{theorem}

In the case $\gamma =0$ we can give an analogous result for the one dimensional case:

\begin{corollary} \label{C:CLT_Regime3}
Let $T>0$ and let the dimension be $d=1$.  Assume that Condition~\ref{A:Assumption1} holds, and that $c(x,y)>0$ for every $x \in \R, y \in \mathcal{Y}$. Then, in the case $i=2$, and $\gamma =0$, the conclusion of Theorem \ref{T:CLT2} holds by setting in the corresponding expressions $\gamma=0$.
\end{corollary}
The proof of this corollary is omitted since, due to periodicity and the condition $c(x,y)>0$, the ODE $\dot{z}_{t}=c(x,z_{t})$ has a unique invariant measure for any $x\in\mathbb{R}$, which then allows the proof of Regime $2$, presented below, to go through with $\gamma=0$.

Before proceeding with the proof of Theorem \ref{T:CLT2}, we mention a useful observation in the remark below.
\begin{remark}\label{R:AdditionalTerms}
Notice that the drift term in the effective equation for $\bar{\eta}$ has an extra term $\bar{J}_{i}(\bar{X}^{i}_{t})$, which is present in the case $\ell_{i}\neq \infty$. This term arises from the fluctuations associated with the term $I^{\epsilon,i}_{t}$ in (\ref{Eq:FluctuationsTerm}). Hence, if $\ell_{i}\neq\infty$, the contribution of this term is not negligible in the limit.
\end{remark}

\begin{proof}[Proof of Theorem~\ref{T:CLT2}]
For each one of the regimes, the proof goes in three steps. First, we deduce a convenient expression for the difference $\Delta^{\epsilon} (t) = X^{\epsilon}_{t} - \bar{X}^{i}_{t}(x_{0})$. Then, in the second step, we use the expression obtained in the last step to prove tightness of the process $\eta^{\epsilon}_{t}=\Delta^{\epsilon}(t)/\beta^{\epsilon}(\ell)$. Finally, in the third step, the limit point for the family of processes $\left( \eta^\eps \right)_{\eps>0}$ is obtained by formulating a martingale problem. We start with Regime 2, and then finalize the proof by proving the case in which the system is in Regime 1. For notational convenience we omit the subscripts $i$, and $x_0$, when no confusion arises, throughout the proof.

Let us first consider Regime $2$, i.e.,  $\gamma \in (0,\infty)$. Upon defining $\Delta^{\epsilon} (t) = X^{\epsilon}_{t} - \bar{X}^{i}_{t}(x_{0})$, we obtain that  $\Delta^{\epsilon} (t)$ is the solution to the equation

\begin{align} \notag
d\Delta^{\epsilon} (t)&=\left[  \frac{\epsilon}{\delta}b\left(  X_{t}^{\epsilon
},\frac{X_{t}^{\epsilon}}{\delta}\right)  +c\left(  X_{t}^{\epsilon}%
,\frac{X_{t}^{\epsilon}}{\delta}\right)  +\epsilon^{\alpha_{1}/2}\Psi^\eps \left(  X_{s}^{\epsilon},\frac{X_{s}^{\epsilon}}{\delta}\right)-\bar{\lambda}_{2}( \bar X _ t )\right]  dt \\
& \qquad +\sqrt{\epsilon}\sigma\left(  X_{t}^{\epsilon},\frac{X_{t}^{\epsilon}}{\delta}\right)dW_{t},\label{eqn: DeltaEqn}
\end{align}
with initial condition $\Delta^{\epsilon} (0)=\eps^{\alpha_2/2} \xi_\eps$. Let us rewrite this equation in terms of $\bar{\lambda}=\bar{\lambda}_{2}$. In order to do so, observe that, since $\bar{\lambda}$ is smooth, Taylor's theorem implies that
\[
\bar{\lambda}(x_1)=  \bar{\lambda}(x_{2}) +D_x \bar{\lambda}(x_{2}) (x_{1}-x_{2}) + Q[\bar \lambda](x_{1},x_{2}), \quad x_{1},x_{2} \in \R^d,
\]
for some function $Q[\bar \lambda]$ such that $|x_1 - x_2|^{-2}Q[\bar \lambda] (x_1,x_2)$ is locally bounded. Then, rewriting~\eqref{eqn: DeltaEqn},
\begin{align} \notag
\Delta^\epsilon(t) &= \eps^{\alpha_2/2}\xi^\eps + \theta^\eps_2 \int_{0}^{t}b\left(  X_{s}^{\epsilon},\frac{X_{s}^{\epsilon}}{\delta}\right)ds  + \int_{0}^{t} D_x \bar{\lambda} ( \bar X_s) \Delta^\eps (s) ds  \\ \notag
& \quad +\int_{0}^{t}\left[\lambda  \left(  X_{s}^{\epsilon},\frac{X_{s}^{\epsilon}}{\delta}\right) - \bar{\lambda}(X_{s}^{\epsilon})\right]  ds
+ \epsilon^{\alpha_{1}/2}\int_{0}^{t} \Psi^\eps \left(  X_{s}^{\epsilon},\frac{X_{s}^{\epsilon}}{\delta}\right)ds
\\
& \quad  + \sqrt{\epsilon}\int_{0}^{t}\sigma  \left(X_{s}^{\epsilon},\frac{X_{s}^{\epsilon}}{\delta}\right) dW_{s}
 + \int_0^t Q[\bar{\lambda}]\parbar{ \bar X_s,X^\eps_s}ds. \label{eqn: delta}
\end{align}
In order to understand the asymptotics of the right hand side of the last display, we need to understand the behavior of the integral term $$I^{\epsilon}_{t}=\int_{0}^{t}\left[\lambda  \left( X_{s}^{\epsilon},\frac{X_{s}^{\epsilon}}{\delta}\right) - \bar{\lambda}(X_{s}^{\epsilon})\right]  ds.$$ For this purpose, apply  It\^{o}'s formula to $\Xi(x,x/\delta)=(\Xi_{1}(x,x/\delta
),\ldots,\Xi_{d}(x,x/\delta))$ with $x=X_{t}^{\epsilon}$. After some algebra, it follows that
\begin{align*}
\delta d \Xi \parbar{ X^\eps _t, \frac{ X^\eps _t }{\delta}  } &= \left[ \mathcal{L}^2_{X^\eps_t} \Xi + \eps^{\alpha_1/2} \Psi^\eps \nabla_y \Xi + \theta^\eps_{2} J_2 \right] \parbar{ X^\eps _t, \frac{ X^\eps _t }{\delta}  } dt \\
& \quad + \sqrt{\eps} \left[ \parbar{ \nabla_y \Xi + \delta \nabla_x \Xi } \sigma \right] \parbar{ X^\eps _t, \frac{ X^\eps _t }{\delta} }dW_t +  R^\eps \parbar{ X^\eps _t, \frac{ X^\eps _t }{\delta}  }dt,
\end{align*}
where, for each $(x,y) \in \R^d \times \mathcal{Y}$, $R^\eps$ is given by
\begin{equation*} 
R^\eps (x,y)= \left[ \eps b \nabla_x \Xi + \delta \parbar{ c  + \eps^{\alpha_1/2} \Psi^\eps  } \nabla_x \Xi + \eps \textrm{tr}\left[\sigma \sigma^T \parbar{ \nabla_x \nabla_y  + \frac{\delta}{2} \nabla_x \nabla_x  }\Xi \right] \right] (x,y).
\end{equation*}
Therefore, taking into account the PDE that $\Xi$ satisfies, we get that
\begin{align} \notag
I^\eps_{t} &=\int_0^t \parbar{\eps^{\alpha_1/2} \Psi^\eps \nabla_y \Xi + \theta^\eps_{2} J_2 }   \left( X_{s}^{\epsilon},\frac{X_{s}^{\epsilon}}{\delta}\right)  ds \\ \notag
& \quad  -\delta\left(\Xi \parbar{ X^\eps _t, \frac{ X^\eps _t }{\delta}  }-\Xi \parbar{ X^\eps _0, \frac{ X^\eps _0 }{\delta}  }\right)
 + \sqrt{\epsilon}\int_{0}^{t} \nabla_{y}\Xi\sigma\parbar{ X^\eps _s, \frac{ X^\eps _s }{\delta}  }dW_s  \\
 & \quad + \int_0^t R^\eps \parbar{ X^\eps _s, \frac{ X^\eps _s }{\delta}  }ds  + \delta\sqrt{\epsilon}\int_{0}^{t} \nabla_{x}\Xi\sigma\parbar{ X^\eps _s, \frac{ X^\eps _s }{\delta}  }dW_s \label{eqn: ieps}
\end{align}
Using the definition of $\Psi_2$ together with~\eqref{eqn: ieps} in~\eqref{eqn: delta} it follows that
\begin{align}
\Delta^\epsilon(t) &=  \eps^{\alpha_2/2}\xi^\eps + \int_{0}^{t} D_x \bar{\lambda} ( \bar X_s) \Delta^\eps (s) ds+\epsilon^{\alpha_{1}/2}\int_{0}^{t} \Psi^\eps_2 \left(  X_{s}^{\epsilon},\frac{X_{s}^{\epsilon}}{\delta}\right)ds\nonumber\\
& \quad + \theta^{\epsilon}_2\int_{0}^{t}J_2 \left( X_{s}^{\epsilon},\frac{X_{s}^{\epsilon}}{\delta}\right)  ds + \sqrt{\epsilon}\int_{0}^{t} \left(I+\nabla_{y}\Xi\right)\sigma\parbar{ X^\eps _s, \frac{ X^\eps _s }{\delta}  }dW_s\nonumber
\\
& \quad + \int_{0}^{t} R^\eps \parbar{ X^\eps _s, \frac{ X^\eps _s }{\delta}  } ds
-\delta\left(\Xi \parbar{ X^\eps _t, \frac{ X^\eps _t }{\delta}  }-\Xi \parbar{ X^\eps _0, \frac{ X^\eps _0 }{\delta}  }\right) \nonumber \\
& \quad +\delta\sqrt{\epsilon}\int_{0}^{t} \nabla_{x}\Xi\sigma\parbar{ X^\eps _s, \frac{ X^\eps _s }{\delta}  }dW_s
+ \int_0^t Q[\bar{\lambda}]\parbar{ \bar X_s,X^\eps_s}ds.\label{Eq:DeltaExpression}
\end{align}
We are now going to use this representation to prove that the family $\xbr{ \eta^\eps }_{ \eps > 0} = \left\{ \Delta^\eps / \beta^\eps (\ell), t \in [0,T] \right \}_{ \eps >0}$ is relatively compact in $C( [0,T]; \R^d)$. To do so we shall use Theorem 8.7 of~\cite{Billingsley1968}, which says that the family of processes $\xbr{ \eta^\eps }_{ \eps > 0}$ is relatively compact in $C( [0,T]; \R^d)$, if there is an $\eps_0$ such that for every $h>0$,
\begin{enumerate}
\item there is a $N_h < \infty$ so that
\[
\Pp \xbr{ \sup_{ t \in [0,T] } |\eta^\eps_t|  > N_h }<h, \quad \eps \in (0,\eps_0), \text{ and}
\]
\item for every $M<\infty$,
\[
\lim_{ r \to 0} \sup_{ \eps \in (0,\eps_0) } \Pp \xbr{ \sup_{ t_1,t_2 \in [0,T], |t_1 -t_2|<r }  |\eta^\eps_{t_1} - \eta^\eps_{t_2} |>h, \sup_{t \in [0,T] } |\eta^\eps_{t}|<M }=0.
\]
\end{enumerate}
We will prove these two points for the family $\eta^\eps = \Delta^\eps / \beta^\eps $.

Before we proceed to prove points (i) and (ii) above, we define some extra notation. Let
\begin{align} \notag
\Theta^{\epsilon}_{x_0} (t)&= \epsilon^{\frac{\alpha_{2}}{2}-m}\Phi_{x_{0}} (t)\xi^{\epsilon} + \epsilon^{\frac{\alpha_{1}}{2}-m}\Phi_{x_{0}} (t)
\int_0^t \left[\Phi_{x_{0}} (s)\right]^{-1}  \Psi^\eps_2 \left(  X_{s}^{\epsilon},\frac{X_{s}^{\epsilon}}{\delta}\right)ds \\
& \quad + \epsilon^{\frac{1}{2}-m}  \Phi_{x_{0}} (t) \int_{0}^{t} \left[\Phi_{x_{0}} (s)\right]^{-1} (I+\nabla_{y}\Xi)\sigma  \left(X_{s}^{\epsilon},\frac{X_{s}^{\epsilon}}{\delta}\right) dW_{s},\label{eqn: thetadef}
\end{align}
so that Duhamel's  principle implies that
\begin{eqnarray}
\Delta^\eps (t)&=&\eps^m \Theta^\eps_{x_0} (t)+ \theta^\eps_2 \Phi_{x_{0}} (t) \int_0^t \left[\Phi_{x_{0}} (s)\right]^{-1}  J_2\left( X_{s}^{\epsilon},\frac{X_{s}^{\epsilon}}{\delta}\right) ds \nonumber\\
& &\quad + R^{\epsilon}_{t}[\Phi]+\Phi_{x_{0}} (t)\int_0^t \left[\Phi_{x_{0}} (s)\right]^{-1}Q[\bar{\lambda}]\parbar{ \bar X_s,X^\eps_s}ds.\label{eqn: Duhamel}
\end{eqnarray}
where the term $R^{\epsilon}_{t}[\Phi]$ is defined as
\begin{align*}
R^{\epsilon}_{t}[\Phi]&=-\delta\Phi_{x_{0}} (t)\left(\Xi \parbar{ X^\eps _t, \frac{ X^\eps _t }{\delta}  }-\Xi \parbar{ X^\eps _0, \frac{ X^\eps _0 }{\delta}  }\right)\nonumber\\
&+\Phi_{x_{0}} (t)\int_{0}^{t}\left[\Phi_{x_{0}} (s)\right]^{-1}R^\eps \parbar{ X^\eps _s, \frac{ X^\eps _s }{\delta}  }ds\nonumber\\
&+\delta\sqrt{\epsilon}\Phi_{x_{0}} (t)\int_{0}^{t}\left[\Phi_{x_{0}} (s)\right]^{-1} \nabla_{x}\Xi\sigma\parbar{ X^\eps _s, \frac{ X^\eps _s }{\delta}  }dW_s.\nonumber
\end{align*}
We have now all the elements we need to show that points (i) and (ii) hold for the family $\Delta^\eps/ \beta^\eps (\ell)$, $\eps >0$. First, due to boundedness of the involved functions, the definition of $R^\eps$, and Doob's inequality for the martingale term, it is easy to see that
\[
\lim_{\epsilon\downarrow 0}\mathbb{E}\left[\sup_{0\leq t\leq T}\left(\left[\beta^{\epsilon}_{i}(\ell)\right]^{-1}R^{\epsilon}_{t}[\Phi]\right)^{2}\right]=0.
\]
Likewise, we will show that the family of processes
\[
\Lambda^\eps_t=\frac{1}{\beta^\eps (\ell) } \Phi_{x_{0}} (t) \int_0^t \left[\Phi_{x_{0}} (s)\right]^{-1}Q[\bar{\lambda}]\parbar{ \bar X_s,X^\eps_s}ds.
\]
also converges to $0$ uniformly on $[0,T]$ in probability as $\eps \to 0$. Once we have these two facts, points (i) and (ii) follow for $\Delta^\eps/\beta^\eps(\ell)$ due to the definition of $\beta^\eps (\ell)$,~\eqref{eqn: Duhamel}, and the boundedness of all functions. Hence, we are just left to prove that
\[
\sup_{t \leq T } |\Lambda^\eps_t | \stackrel{\Pp}{\longrightarrow} 0, \quad \eps \to 0.
\]
Let $\nu \in(1/2,1)$ and set
\[
\tau^\eps (\nu) = \inf \left \{t: \left | X^\eps_t - \bar X_t \right | > (\beta^\eps)^\nu    \right \}.
\]
Using the quadratic decay of $Q[\bar \lambda]$, and the fact that $2 \nu > 1 $, we see that $$\lim_{\eps \to 0} \sup_{t \leq T \wedge \tau^\eps(\nu) } |\Lambda^\eps_t |=0,$$ with probability $1$ . As a consequence, we are left to show that $\Pp \{ T < \tau^\eps ( \nu ) \} \to 1$. To do so, note that~\eqref{eqn: Duhamel} implies that if $\tau^\eps (\nu) < T$,
\begin{align*}
1 &= (\beta^\eps)^{ - \nu } \sup_{ t \leq T \wedge \tau^\eps ( \nu ) } | \Delta^\eps (t) | \\
& \leq (\beta^\eps)^{1-\nu} C_1  + C_2 (\beta^\eps)^{\nu},
\end{align*}
for some random variables $C_1, C_2 < \infty$, $\Pp$-a.s. Hence, since the r.h.s of the last display converges to $0$, it follows that $\lim_{\eps \to 0} \Pp \{ \tau^\eps (\nu) < T \}=0$, which implies the precompactness of the family $\{\eta^{\eps}=\Delta^\eps / \beta^\eps ( \ell),\eps>0\}$. Clearly, the tightness of the family $\{X^{\epsilon}_{\cdot}, \epsilon>0\}$ implies tightness of the pair $\left\{(\eta^{\epsilon}_{\cdot},X^{\epsilon}_{\cdot}),\epsilon>0 \right\}$.

Now that we know that the family $\left\{(\eta^{\epsilon}_{\cdot},X^{\epsilon}_{\cdot}),\epsilon>0 \right\}$ is precompact, we are left to identify its limit. We shall use (\ref{Eq:DeltaExpression}) and the martingale problem formulation.
An inspection of (\ref{Eq:DeltaExpression}) shows that the terms in its third and fourth line
are of lower order compared to the other terms and thus should vanish in the limit. Let us make this now rigorous.

Consider, the process
\begin{equation*}
\zeta^{\epsilon}_{t}=\eta^{\epsilon}_{t}+\delta \left(\beta^{\epsilon}(\ell)\right)^{-1}\left(\Xi \parbar{ X^\eps _t, \frac{ X^\eps _t }{\delta}  }-\Xi \parbar{ X^\eps _0, \frac{ X^\eps _0 }{\delta}  }\right)
\end{equation*}
and let $\phi\in C^{2}_{b}(\mathbb{R}^{d})$. We identify the limit using the martingale problem approach. Write for simplicity $\beta=\beta^{\epsilon}(\ell)$. By It\^{o} formula we have
\begin{align} \notag
\phi(\zeta^\epsilon_{t}) &=  \phi(\beta^{-1}\eps^{\alpha_2/2}\xi^\eps) + \int_{0}^{t} D_x \bar{\lambda} ( \bar X_s) \eta^\eps (s) \nabla\phi(\zeta^{\eps}_{s}) ds+
\frac{\theta^{\epsilon}}{\beta}\int_{0}^{t}J \left( X_{s}^{\epsilon},\frac{X_{s}^{\epsilon}}{\delta}\right)  \nabla\phi(\zeta^{\epsilon}_{s}) ds\\
&\hspace{0.3cm}+\frac{\epsilon}{\beta^{2}}\frac{1}{2}\int_{0}^{t}\textrm{tr}\left[ \nabla\nabla\phi(\zeta^{\epsilon}_{s}) \left(I+\nabla_{y}\Xi\right)\sigma\sigma^{T}\left(I+\nabla_{y}\Xi\right)^{T} \right]\parbar{ X^\eps _s, \frac{ X^\eps _s }{\delta}  }ds\nonumber\\
&\hspace{0.3cm}+\frac{\sqrt{\epsilon}}{\beta}\int_{0}^{t} \nabla\phi(\zeta^{\epsilon}_{s}) \left(I+\nabla_{y}\Xi\sigma\right)\parbar{ X^\eps _s, \frac{ X^\eps _s }{\delta}  }dW_s\nonumber\\
&\hspace{0.3cm}+\int_{0}^{t}\left[\left(\frac{\epsilon^{\alpha_{1}/2}}{\beta} \Psi^\eps_{2} +\frac{1}{\beta}R^{\epsilon}\right)\parbar{ X^\eps _s, \frac{ X^\eps _s }{\delta}  }+\frac{1}{\beta}Q[\bar{\lambda}]\parbar{ \bar X_s,X^\eps_s}\right]\nabla\phi(\zeta^{\epsilon}_{s})ds+\nonumber\\
&\quad +\frac{\delta^{2}\epsilon}{\beta^{2}}\frac{1}{2}\int_{0}^{t}\textrm{tr}\left[ \nabla\nabla\phi(\zeta^{\epsilon}_{s})\nabla_{x}\Xi\sigma(\nabla_{x}\Xi\sigma)^{T}\parbar{ X^\eps _s, \frac{ X^\eps _s }{\delta}  }\right]ds\nonumber\\
&\quad+\frac{\delta\sqrt{\epsilon}}{\beta}\int_{0}^{t} \nabla\phi(\zeta^{\epsilon}_{s})\nabla_{x}\Xi\sigma\parbar{ X^\eps _s, \frac{ X^\eps _s }{\delta}  }dW_s \label{Eq:PhiExpression}
\end{align}

It is clear from (\ref{Eq:DeltaExpression}) that there are different cases to consider, depending on the order that the different terms go to zero. For the sake of presentation, we shall only study in detail the case $m=1/2$ and $\ell\neq 0$. In this case the limiting process is a solution to a stochastic differential equation. The other cases follow similarly. Without loss of generality, let us simplify the problem more and assume that $a_{1},a_{2}>1$. Then, we actually have $\beta=\beta^{\epsilon}(\ell)=\sqrt{\epsilon}$ and our target is to prove that for any $0\leq s\leq t\leq T$
\begin{eqnarray}
& &\lim_{\epsilon\downarrow 0}\mathbb{E}_{\eta^{\epsilon}_{0}}\left[\phi(\eta^{\epsilon}_{t})-\phi(\eta^{\epsilon}_{s})-\int_{s}^{t}\left[\left(D_x \bar{\lambda} ( \bar X_r) \eta^\eps (r)+\ell^{-1}\bar{J} \left( X_{r}^{\epsilon}\right)\right)\nabla\phi(\eta^{\eps}_{r})\right.\right.\nonumber\\
& &\hspace{5.3cm}\left.\left.+ \frac{1}{2}\textrm{tr}\left[ \nabla\nabla\phi(\eta^{\epsilon}_{r}) q\left( X_{r}^{\epsilon}\right) \right]\right]dr\Big | \mathcal{F}_{s}\right]=0\label{Eq:Martingale}
\end{eqnarray}
This follows directly upon rewriting the left hand side of (\ref{Eq:Martingale}) using (\ref{Eq:PhiExpression}). In particular, note that the following holds:
\begin{enumerate}
\item{Due to boundedness of $\Xi$
\begin{equation*}
\lim_{\epsilon\downarrow0}\delta\beta^{-1}\mathbb{E}\sup_{0\leq t\leq T}\left|\Xi\parbar{ X^\eps _t, \frac{ X^\eps _t }{\delta}  }\right|=0.
\end{equation*}
}
\item{Due to boundedness of the coefficients and the bounds on the derivatives of the auxiliary function $\Xi$ (see Theorem 3.3.4 in \cite{BLP}), we have that
\begin{equation*}
\lim_{\epsilon\downarrow0}\mathbb{E}\sup_{0\leq t\leq T}\int_{0}^{t}\left|\left[\left(\frac{\epsilon^{\alpha_{1}/2}}{\beta} \Psi^\eps_{2} +\frac{1}{\beta}R^{\epsilon}\right)\parbar{ X^\eps _s, \frac{ X^\eps _s }{\delta}  }\right]\nabla\phi(\zeta^{\epsilon}_{s})\right|ds=0.
\end{equation*}
Similarly to the argument used to prove tightness, we also have that
\[
\lim_{\eps\downarrow 0}\mathbb{E}\int_{s}^{t}(\beta^\eps)^{-1}Q[\bar {\lambda}]\parbar{ \bar X_r,X^\eps_r}\nabla\phi(\zeta^{\epsilon}_{r})dr= 0.
\]
Moreover, it is clear that there is also a constant $C<\infty$ such that for the remaining Riemann integrals  in (\ref{Eq:PhiExpression})
\begin{equation*}
\mathbb{E}\sup_{0\leq t\leq T}\int_{0}^{t}\left|\Gamma \parbar{ X^\eps _s, \frac{ X^\eps _s }{\delta}  }\right|ds\leq C.
\end{equation*}
}
\item{Due to boundedness of the coefficients and the bounds on the derivatives of the auxiliary function $\Xi$ \cite{BLP}, there is a constant $C<\infty$ such that, using Doob's martingale inequality, for any stochastic integral in (\ref{Eq:PhiExpression})
\begin{equation*}
\mathbb{E}\sup_{0\leq t\leq T}\left|\int_{0}^{t} \hat \Gamma\parbar{ X^\eps _s, \frac{ X^\eps _s }{\delta}  }dW_{s}\right|^{2}\leq C
\end{equation*}
}
\item{By assumption $\beta=\sqrt{\epsilon}$ and $\theta/\beta\rightarrow \ell^{-1}$ as $\epsilon\downarrow 0$.}
\end{enumerate}

The validity of (\ref{Eq:Martingale}) together with tightness of the pair $\left\{(\eta^{\epsilon}_{\cdot},X^{\epsilon}_{\cdot}),\epsilon>0 \right\}$ and uniqueness of the limiting martingale problem imply the claim.

Let us now concentrate on Regime 1; that is, $\gamma=\infty$. Consider the solution to the cell problem $\chi=(\chi_{1},\ldots,\chi_{d})$, which is periodic in every coordinate direction in $y$ and satisfies
\begin{equation}
\mathcal{L}_{x}^{1}\chi_{l}(x,y)=-b_{l}(x,y),\quad\int_{\mathcal{Y}}%
\chi_{l}(x,y)\mu(dy|x)=0,\hspace{0.1cm} l=1,...,d.\nonumber
\end{equation}
By applying  It\^{o}'s formula to $\chi(x,x/\delta)=(\chi_{1}(x,x/\delta
),\ldots,\chi_{d}(x,x/\delta))$ with $x=X_{t}^{\epsilon}$, we can reduce the problem to the previous case. Indeed, by It\^{o}'s formula, the cell problem formulation, and Definition~\ref{Def:ThreePossibleFunctions} it follows that
\begin{align*}
d \chi \parbar{ X^\eps _t, \frac{ X^\eps _t }{\delta}  } &= \left[ \parbar{\frac{\eps}{\delta}b + c} \cdot \nabla_x + \frac{1}{\delta} c\cdot \nabla_y + \frac{\eps}{2} \textrm{tr}\left[\sigma \sigma^T  \nabla^{2}_x\right] + \right.\\
&\left.\quad +\frac{\eps}{ \delta} \textrm{tr}\left[\sigma \sigma^T \nabla_x \nabla_y\right]  + \eps^{\alpha_1/2} \Psi^\eps \cdot \parbar{ \nabla_x + \frac{1}{\delta} \nabla_y }\right] \chi \parbar{ X^\eps _t, \frac{ X^\eps _t }{\delta} }dt \\
& \quad + \frac{\eps }{\delta^2} \mathcal{L}_{X^\eps_t}^{1} \chi \parbar{ X^\eps _t, \frac{ X^\eps _t }{\delta} } dt
+ \sqrt{\eps} \left[\parbar{ \nabla_x + \frac{1}{\delta} \nabla_y }\sigma \right] \chi \parbar{ X^\eps _t, \frac{ X^\eps _t }{\delta} }dW_t
\end{align*}
Hence, recalling the cell problem (\ref{Eq:CellProblem}) we have
\begin{align*}
\frac{\eps }{\delta} b \parbar{ X^\eps _t, \frac{ X^\eps _t }{\delta} } dt& =\left [ \parbar{\eps b + \delta c} \cdot \nabla_x +  c\cdot \nabla_y + \frac{\eps}{2}\delta \textrm{tr}\left[\sigma \sigma^T  \nabla^{2}_x \right] +\right.\\
 &\left. \qquad +\eps \textrm{tr}\left[\sigma \sigma^T \nabla_x \nabla_y\right]  + \delta\eps^{\alpha_1/2} \Psi^\eps \cdot \parbar{ \nabla_x + \frac{1}{\delta} \nabla_y }\right] \chi \parbar{ X^\eps _t, \frac{ X^\eps _t }{\delta} }dt \\
& \quad
+ \sqrt{\eps}\delta \left[  \parbar{ \nabla_x + \frac{1}{\delta} \nabla_y }\sigma \right] \chi \parbar{ X^\eps _t, \frac{ X^\eps _t }{\delta} }dW_t-d \chi \parbar{ X^\eps _t, \frac{ X^\eps _t }{\delta}  }.
\end{align*}
Using this in~\eqref{eqn: DeltaEqn}, rearranging terms, and proceeding as in Regime 2, we obtain (with $\lambda=\lambda_{1},\bar{\lambda}=\bar{\lambda}_{1}$)
\begin{align*}
d \Delta^{\epsilon} (t)  &= \left[D_x \bar{\lambda} ( \bar X_t) \Delta^\eps (t)
+ \parbar{ \lambda \parbar{ X^\eps _t, \frac{ X^\eps _t }{\delta} }  - \bar{\lambda} \left( X^\eps _t\right)   }
+ \epsilon^{\alpha_{1}/2} (I+\nabla_{y}\chi) \Psi^{\epsilon}_1 \parbar{ X^\eps _t, \frac{ X^\eps _t }{\delta} }  \right]dt \\
& \quad +\sqrt{\epsilon} (I+\nabla_{y}\chi)\sigma\parbar{ X^\eps _t, \frac{ X^\eps _t }{\delta} } dW_{t} +\mathcal{G}^\eps \chi \parbar{ X^\eps _t, \frac{ X^\eps _t }{\delta} }dt\\
&\quad + R_{1}^{\epsilon}(X^\eps_t, \frac{X^\eps_t}{\delta} ) dW_t
+Q[\bar{\lambda}]\parbar{ \bar X_t,X^\eps_t}dt- \delta d \chi \parbar{ X^\eps _t, \frac{ X^\eps _t }{\delta}  } ,
\end{align*}
with $\Delta^\eps(0) =\eps^{ \alpha_2/2} \xi^\eps$. Here, we have defined
\begin{align*}
\mathcal{G}^\eps\chi(x,y) &= \left[ \parbar{\eps b+ \delta c }\cdot \nabla_x \chi +  \eps \textrm{tr} \left[ \sigma \sigma^T \parbar{ \frac{\delta }{2} \nabla_x \nabla_x  + \nabla_x \nabla_y } \chi \right]+ \delta\eps^{\alpha_1/2} \Psi^\eps(x,y)  \nabla_x \chi \right] \parbar{x,y }\\
R_{1}^{\epsilon}(x,y)&= \sqrt{\epsilon}\delta  \nabla_x \chi (x,y)\sigma .
\end{align*}

As in Regime $2$, we need to understand the behavior of the term
$$I^{\epsilon}_{t}=\int_{0}^{t}\left[\lambda  \left( X_{s}^{\epsilon},\frac{X_{s}^{\epsilon}}{\delta}\right) - \bar{\lambda}(X_{s}^{\epsilon})\right]  ds.$$ For this purpose,
apply  It\^{o}'s formula to $\Xi(x,x/\delta)=(\Xi_{1}(x,x/\delta
),\ldots,\Xi_{d}(x,x/\delta))$ with $x=X_{t}^{\epsilon}$. Taking into account the PDE that $\Xi$ satisfies, we get
\begin{align*}
\int_{0}^{t}\left[\lambda \left( X_{s}^{\epsilon},\frac{X_{s}^{\epsilon}}{\delta}\right) -\bar{\lambda}(X_{s}^{\epsilon})\right]  ds &=
\frac{\delta}{\epsilon}\int_{0}^{t}c\nabla_{y}\Xi\left( X_{s}^{\epsilon},\frac{X_{s}^{\epsilon}}{\delta}\right)ds +R^{\epsilon}_{2}(t) \nonumber\\
&\quad -\frac{\delta^{2}}{\epsilon}\left(\Xi \parbar{ X^\eps _t, \frac{ X^\eps _t }{\delta}  }-\Xi \parbar{ X^\eps _0, \frac{ X^\eps _0 }{\delta}  }\right)
\end{align*}
where the lower order term $R^{\epsilon}_{2}(t)$ is
\begin{align*}
R^{\epsilon}_{2}(t)&=
\frac{\delta^{2}}{\epsilon}\int_{0}^{t}\left(c+\epsilon^{\alpha_{1}/2}\Psi^{\epsilon}\right)\nabla_{x}\Xi\parbar{ X^\eps _s, \frac{ X^\eps _s }{\delta}  }ds+\frac{\delta}{\epsilon}\epsilon^{\alpha_{1}/2}\int_{0}^{t}\Psi^{\epsilon}\nabla_{y}\Xi\parbar{ X^\eps _s, \frac{ X^\eps _s }{\delta}  }ds\nonumber\\
&\quad+\delta\int_{0}^{t}\left(b\nabla_{x}\Xi+\textrm{tr}\left[\sigma \sigma^T \nabla_{x}\nabla_{y}\Xi\right]\right)\parbar{ X^\eps _s, \frac{ X^\eps _s }{\delta}  }ds+
\delta^{2}\frac{1}{2}\int_{0}^{t}\textrm{tr}\left[\sigma \sigma^T \nabla_{x}\nabla_{x}\Xi\right]\parbar{ X^\eps _s, \frac{ X^\eps _s }{\delta}  }ds\nonumber\\
&\quad+\frac{\delta^{2}}{\epsilon}\sqrt{\epsilon}\int_{0}^{t} \nabla_{x}\Xi\sigma\parbar{ X^\eps _s, \frac{ X^\eps _s }{\delta}  }dW_s
+\frac{\delta}{\epsilon}\sqrt{\epsilon}\int_{0}^{t} \nabla_{y}\Xi\sigma\parbar{ X^\eps _s, \frac{ X^\eps _s }{\delta}  }dW_s\nonumber
\end{align*}

Hence, we get that
\begin{align*}
 \Delta^{\epsilon} (t)  &= \eps^{\alpha_{2}/2}\xi^{\epsilon}+\int_{0}^{t}D_x \bar{\lambda} ( \bar X_s) \Delta^\eps (s)ds
+ \frac{\delta}{\epsilon}\int_{0}^{t}c\nabla_{y}\Xi\left( X_{s}^{\epsilon},\frac{X_{s}^{\epsilon}}{\delta}\right)ds\\
& \quad + \epsilon^{\alpha_{1}/2} \int_{0}^{t} (I+\nabla_{y}\chi)\Psi^{\epsilon}\parbar{ X^\eps _s, \frac{ X^\eps _s }{\delta} }  ds
 +\sqrt{\epsilon} \int_{0}^{t}(I+\nabla_{y}\chi)\sigma\parbar{ X^\eps _s, \frac{ X^\eps _s }{\delta} } dW_{s}\\
  & \quad +\int_{0}^{t}\mathcal{G}^\eps \chi \parbar{ X^\eps _s, \frac{ X^\eps _s }{\delta} }ds + \int_{0}^{t}R_{1}^{\epsilon}(x^\eps_s, \frac{x^\eps_s}{\delta}) dW_s+R^{\epsilon}_{2}(t)
  +\int_{0}^{t}Q[\bar{\lambda}]\parbar{ \bar X_s,X^\eps_s}ds\\
&\quad
- \delta \left[\chi \parbar{ X^\eps _t, \frac{ X^\eps _t }{\delta}  }- \chi \parbar{ X_0, \frac{ X_0 }{\delta}  }\right]
-\frac{\delta^{2}}{\epsilon}\left(\Xi \parbar{ X^\eps _t, \frac{ X^\eps _t }{\delta}  }-\Xi \parbar{ X^\eps _0, \frac{ X^\eps _0 }{\delta}  }\right),
\end{align*}
Using these expressions, tightness follows after using Duhamel principle using the same arguments as for Regime $2$. The identification and uniqueness of the limiting point follows by the martingale problem by considering the process
\begin{equation*}
\zeta^{\epsilon}_{t}=\eta^{\epsilon}_{t}+\frac{\delta^{2}}{\epsilon}\beta^{-1} \left(\Xi \parbar{ X^\eps _t, \frac{ X^\eps _t }{\delta}  }-\Xi \parbar{ X^\eps _0, \frac{ X^\eps _0 }{\delta}  }\right)+\delta\beta^{-1} \left(\chi \parbar{ X^\eps _t, \frac{ X^\eps _t }{\delta}  }-\chi \parbar{ X^\eps _0, \frac{ X^\eps _0 }{\delta}  }\right)
\end{equation*}
and we applying It\^{o} formula to a test function $\phi\in C^{2}_{b}(\mathbb{R}^{d})$ with stochastic process $\zeta^{\epsilon}_{t}$. Using the resulting expression, the claim follows by the martingale problem formulation as in Regime $2$. Thus, the details are omitted.
\end{proof}

\section{On Asymptotics for the Exit Time and Exit Location}\label{S:Exit}

Let us state the main result in regards with the correction to the exit. First we describe our assumptions on the joint geometry of the vector field $\bar{\lambda}$ and the surface $M$.
We define
\[
T^i=\inf \left \{t>0: \bar X^i_t (x_0) \in M \right \},
\]
and assume that $0<T^i<\infty$, for each $i=1,2$. Also, for each $i=1,2$, we denote $z^i = \bar X^i _{ T^i } (x_0) \in M$ and assume that $\bar \lambda_i (z)$ does not belong to the tangent hyperplane $T_{z^i} M$ of $M$ at $z^i$. In other words, we assume that the positive orbit of $x_0$ under the vector field $\bar \lambda_i$ crosses $M$ transversally.

 Let us introduce a local basis around the exit point $z^i$ in order to express the correction. Given $z\in \R^d$, let $T_z M$ be the tangent space of $M$ at $z$. For $i=1,2$, define the projections $\pi^i: \R^d \to \R$ and $\pi_M^i :\R^d \to T_z M$ by
\[
 v=\pi^i v \cdot \bar\lambda_i(z)+ \pi_M^i v, \quad v \in \R^d.
\]
That is, $\pi^i$ is the (algebraic) projection onto $\Span(\bar \lambda_i (z^i ))$ along $T^i M$ and $\pi_M^i$ is the (geometric) projection onto $T_{z^i} M$ along $\Span( \bar \lambda_i(z^i))$.

We have now all the elements necessary to state the main theorem:

\begin{theorem} \label{thm: Main} With the notation of Theorem \ref{T:CLT2}, let  $\ell_{i} < \infty$ and assume additionally that $\eps^{-\zeta} \theta^{\epsilon}_{i} \to 1$,
for some $\zeta >0$, as $\epsilon,\delta\downarrow 0$.
Then, in the setting introduced above and under these assumptions, for $i=1,2$,
\begin{equation}
 \frac{1}{\beta^\eps(\ell)} (\tau^\eps-T^i, X^\eps_{\tau^\eps}-z^i){\to} \parbar{-\pi^i \bar{\eta}^i_T (\ell) , \pi^i_M \bar{\eta}^i_T(\ell) }. \label{eq:main_convergence_statement}
\end{equation}
in distribution as $\eps \to 0$.
\end{theorem}

\subsection{Proof of Theorem~\ref{thm: Main} }
The proof will be given in several steps. We use some of the ideas of the proof of Theorem 1 in [1], even though, due to the averaging effects, several new ingredients are needed.

Under the assumptions of the theorem, note that $\beta^\eps (  \ell) = \eps^\beta$, where
\[
\beta = \zeta \one \parbar{ \ell =0} + m \one \parbar{ \ell \ne 0}.
\]
It follows that in Regime $i=2$, $\beta=\min \left\{ \zeta,1/2, \alpha_1/2, \alpha_2/2\right\}$.

The following corollary of Theorem~\ref{T:CLT2} is essential in our proof. After the proof of this corollary, we will restate Theorem~\ref{T:CLT2} in a more convenient way for the propose of this proof. Let us start with the corollary:

\begin{corollary} \label{cor: ExitCorrection}
Under the same assumptions as in Theorem~\ref{T:CLT2}, if, as $\eps \to 0$, $t^\eps\to0$, $  t^\eps  (\eps/\delta - \gamma)\eps^{-m}  \to 0,$ and $m=\alpha_2/2$, then
\begin{equation} \label{eqn: ProbConv}
\sup_{ t\leq t^\eps } \left | \eps^{-m} \parbar{ X^{\epsilon}_{t} -  \bar X^i_t } - \xi^{0} \right | \PpConv 0 ,\quad \eps\to0.
\end{equation}
\end{corollary}
\begin{proof}
First, let us focus on the case $i=2$. In this case, observe that, since $m=\alpha/2$, from~\eqref{eqn: thetadef} we have
\[
\eps^m \Theta^\eps_{x_0} (t) = \eps^m \Phi_{x_0} (t) \xi^\eps + \eps^m r^\eps_{x_0} (t),
\]
where $r^\eps_{x_0}$ is a process such that for any family $(s^\eps)_{\eps >0}$ such that $s^\eps\to 0$, it follows that
\[
\sup_{ t\in [0,s^\eps] } |r^\eps_{x_0} (t)| \PpConv0,\quad \eps\to0.
\]
Hence, from ~\eqref{eqn: Duhamel} it follows that
\begin{align*}
\eps^{-m}\Delta^\eps (t)  &= \Phi_{x_0} (t) \xi^\eps + r^\eps_{x_0} (t) \nonumber\\
&+ \eps^{-m} \theta^\eps \Phi_{x_{0}} (t) \int_0^t \left[\Phi_{x_{0}} (s)\right]^{-1}  \left(b(I+\nabla_{y}\Xi)+\frac{1}{2}\textrm{tr}\left[\sigma \sigma^T \nabla_{y}\nabla_{y}\Xi\right]\right)  \left( X_{s}^{\epsilon},\frac{X_{s}^{\epsilon}}{\delta}\right) ds \nonumber\\
& \quad +\eps^{-m} R^{\epsilon}_{t}[\Phi]+ \eps^{-m}  \Phi_{x_{0}} (t)\int_0^{t} \left[\Phi_{x_{0}} (s)\right]^{-1}Q[\bar{\lambda}]\parbar{ \bar X_s,X^\eps_s}ds,
\end{align*}
which implies the result since $\theta^\eps t^\eps/ \eps^m \to0$, and, as in the proof of Theorem~\ref{T:CLT2}, $R^\eps[\Phi]$ is of order $\eps$, and $Q[\bar \lambda]$ of order $(\beta^\eps)^2$ as $\eps \to 0$. For Regime 1, the result follows analogously from the expression obtained for $\Delta^\eps$ in this case.
\end{proof}

We now give a useful restatement of Theorem~\ref{T:CLT2}. Before doing so, some additional notation is needed. For $i=1,2$, let
\begin{align}
\Theta^{i}_{x_0} (t) &= \ONE \parbar{m=\alpha_2/2} \Phi_{x_{0}}^i (t)\xi^{0} \notag  \\
& \quad   + \ONE \parbar{m=\alpha_1/2} \Phi_{x_{0}}^i (t)
\int_0^t \left[\Phi_{x_{0}}^i (s)\right]^{-1} \bar \Psi_{i}^i \left(  \bar X_{s}^{i} \right)ds \notag \\ 
 &\quad + \ONE \parbar{m=1/2}  \Phi_{x_{0}}^i (t) \int_{0}^{t} \left[\Phi_{x_{0}}^i (s)\right]^{-1} \bar{q}^{1/2}_i( \bar X_s^{i} )  dW_{s} , \quad t \geq 0. \label{eqn: thetadef_0}
\end{align}
and
\begin{equation*}
H^{i}_{x_0} (t)= \Phi_{x_{0}}^i (t)
\int_0^t \left[\Phi_{x_{0}}^i (s)\right]^{-1} \bar J_{i} \left(  \bar X_{s}^{i} \right)ds.
\end{equation*}
We have the following result:
\begin{theorem} \label{L:Lemma5_regimes2_3}
 Assume Regime $i=1,2$ and let Conditions \ref{A:Assumption1}-\ref{A:Assumption2} holding.
Consider the stochastic process
\begin{align}
\bar{\eta}_{t}^i (\ell_{i}) =  \Theta^i_{x_0} (t) \one \left(\ell_{i} \ne 0 \right)+  H^i_{x_0} (t)   \left[ \ell^{-1}_{i} \one (\ell_{i} \in(0,\infty] ) + \one (\ell_{i} =0 ) \right ].   \label{eqn: Eta23}
\end{align}

Then, for each $\eps > 0$, there is a process $\eta^\eps ( \ell_{i})$, such that
\[
X^{\epsilon}_t =  \bar X^{i}_{t} + \beta^\eps (\ell_{i}) \eta^{\eps}_t (\ell_{i})
\]
holds a.e. $t>0$ with probability $1$, and $\eta^\eps ( \ell_{i}) \to \bar{\eta}^i (\ell_{i})$, as $\eps \to 0$, in distribution in $\mathcal{C}\left([0,T];\mathbb{R}^{d}\right)$.
\end{theorem}

The proof of Theorem~\ref{L:Lemma5_regimes2_3} follows essentially by Theorem \ref{T:CLT2} and Duhamel's principle.  The details are omitted.

The following Lemma 4.4, allows us to rewrite $X^{\epsilon}_{T-t}$ and $X^{\epsilon}_{T+t}$  in appropriate forms.

\begin{lemma} \label{lemma: X_before_after}
Let $\omega \in (\beta/2,\beta)$. Then, for each $i=1,2,$  there are two a.s.-continuous stochastic processes $\Upsilon^{\eps,\pm,i}$ such that
\[
\sup_{t\in[0,\eps^\omega]} |\Upsilon^{\eps,\pm,i}_t| \stackrel{\Pp}{\longrightarrow} 0, \quad \eps \to 0,
\]
and a.e. $t\in[0,\eps^\omega]$ a.s
\begin{equation}
X^\eps_{ T^i-t} = z^i -  t\bar \lambda_i (z^i) +\eps^\beta \parbar { \eta^\eps_{T-t} ( \ell)  + \Upsilon^{\eps, -,i}_ t } \label{eqn: X_before}
\end{equation} and
\begin{equation}
\tilde X^\eps_{T^i+t}= z^i+t\bar \lambda_i (z^i)+\eps^\beta \parbar{ \tilde \eta^\eps_t (\ell) + \Upsilon^{\eps, +,i}_ t  }\label{eqn: X_after},
\end{equation}
where $\tilde \eta^\eps ( \ell)$ is a stochastic process such that for any $\omega \in (\beta/2,\beta),  \ell \in [0,\infty]$, it follows that
\begin{equation} \label{eqn: tildeetaconv}
 \tilde \eta^\eps _ { \cdot } ( \ell ) \rightarrow \bar{\eta}^i_\cdot  (\ell), \quad \eps\to0, \textrm{ in distribution in }\mathcal{C} ( [0,S];\mathbb{R}^{d} ) \textrm{ for any } S>0.
\end{equation}
\end{lemma}

\begin{proof}
 We will omit reference to the regime $i$ when no confusion arises. Let $\bar X_t(x)$ ( $\bar X_{-t} (x)$) be the positive (negative) orbit of $x_0$ under the flow $\bar \lambda$. Note, in particular, that
\[
\frac{d}{dt} \bar X_{ \pm t} ^i (x)= \pm \bar \lambda_i ( \bar X_{\pm t} ^i(x)  ), \quad \bar X_0^i (x) = x,
\]
for every $x\in \R^d$.

Due to Strong Markov property and Lemma~\ref{L:Lemma5_regimes2_3} the process $\tilde X^\eps_t = X^\eps_{ t + T}$ satisfies the SDE~\eqref{Eq:LDPandA1} with respect to the Brownian Motion $\tilde W(t)=W(t+T)-W(T)$ and initial condition $ \tilde X^\eps_ 0 = z+\eps^\beta \eta^\eps_T (\ell)$.
Hence, we can apply Lemma~\ref{L:Lemma5_regimes2_3} to this shifted equation, to obtain
 \[
\tilde X^\eps_t= \bar X_t^i (z^i)+\eps^{\tilde \beta} \tilde \eta^\eps_t (\ell),
 \]
where $\tilde\beta = \zeta$ if $i=2$, $ \tilde m =\min \{ \beta, 1/2, \alpha_1/2  \} > \zeta$, and $\tilde \beta= \tilde m$ in other case. The process $\tilde \eta^\eps$ is such that  $\tilde \eta^\eps (\ell) \to \tilde \tilde{\eta}^i (\ell)$, as $\eps\to 0$, in distribution in $\mathcal{C} ( [0,S];\mathbb{R}^{d} )$ for any $S>0$, where $\tilde \eta^i$ is defined by
\begin{align}
\tilde \eta_{t}^i(\ell) &=  \tilde \Theta_{z}^i(t) \one \left(\ell \ne 0 \right)
  + \left[ \ell^{-1} \one (\ell \in(0,\infty) ) + \one (\ell =0 ) \right ]  H_{z}^i (t)   \label{eqn: EtaNew},
\end{align}
where
\begin{align*}
\tilde \Theta_{z}^i (t) &=   \Phi_{z}^i (t) \eta_T^i (\ell)  \\
& \quad  + \ONE \parbar{\tilde m=\alpha_1/2} \Phi_{z}^i (t)
\int_0^t \left[\Phi_{z}^i (s)\right]^{-1} \bar \Psi \left(  \bar X_{s}^i(z) \right)ds \notag \\ 
& \quad + \ONE \parbar{\tilde m=1/2}  \Phi_{z}^i (t) \int_{0}^{t} \left[\Phi_{z}^i (s)\right]^{-1} q^{1/2}_i( \bar X_s^i (z) )  dW_{s}, \label{eqn: thetadef_0}
\end{align*}
and $\ell = \lim_{ \eps \to 0 } \eps^{\tilde m } \parbar{  \eps / \delta-\gamma }^{-1}$. Note that $\ell = 0$ if and only if $m > \zeta$. In this case, $\beta=\zeta \leq m \leq \min \{ 1/2,\alpha_1/2 \}$. Hence, in any case, $\tilde m= \beta$, and, from the definition of $\tilde \beta$, $\tilde \beta =\beta$.

Using these results, it follows that
\begin{align*}
X^\eps_{T+t} = z^i+ t \bar\lambda (z^i) + \eps^\beta \parbar{ \tilde \eta^\eps_t (\ell) + \Upsilon^{\eps, +,i}_ t  },
\end{align*}
where $\Upsilon^{\eps, +,i}_ t = \eps^{ - \beta } \parbar {  \bar X_t^i (z^i) - \parbar{ z^i+ t \bar\lambda (z^i)} } $.  Hence to show that~\eqref{eqn: X_after} follows, we have to prove the convergence of $\Upsilon^{\eps, +,i}$ towards zero and~\eqref{eqn: tildeetaconv}. Let us start with the former,  to show that $\Upsilon^{\eps,+,i}$ converges towards $0$, use Lemma 6 in~\cite{SergioBakhtin2011} to find two positive constants $C_1$ and $C_2$ such that for any $t>0$, and $x \in \R^d$,
\begin{equation}
\sup_{s\leq t}\left|  \bar X_{\pm t}^i (x) -\left( z^i \pm t \bar \lambda(z) \right)\right| \leq C_1 e^{C_2 t }( t |x-z^i| + t^2).  \label{eqn: FlowClose}
\end{equation}
This implies the convergence to $0$ of $\Upsilon^{\eps, +,i}$. We are left to prove~\eqref{eqn: tildeetaconv}. In order to so, recall that $\tilde m = \beta$, and Corollary~\ref{cor: ExitCorrection} implies~\eqref{eqn: tildeetaconv}.

To get~\eqref{eqn: X_before}, note that, for $ t \in [0,T]$, Lemma~\ref{L:Lemma5_regimes2_3}, Equation~\eqref{eqn: FlowClose}, and the fact that $\bar X_{T-t}^i (x_0) = \bar X^i_{-t} (z^i)$ imply that
\begin{align*}
X^\eps_{ T-t} &= \bar X_{T-t}^i (x_0) + \eps^\beta  \eta^\eps_{T-t} ( \ell) \\
&= \bar X_{-t} ^i(z^i) + \eps^\beta  \eta^\eps_{T-t} ( \ell) \\
&= z^i -  t\bar \lambda_i (z^i) +\eps^\beta \parbar { \eta^\eps_{T-t} ( \ell)  + \Upsilon^{\eps, -,i}_ t },
\end{align*}
where $\Upsilon^{\eps, -,i}_ t = \eps^{ - \beta } \parbar {  \bar X_{-t}^i (z^i) - \parbar{ z^i - t \bar\lambda_i(z^i)} } $ converges to $0$ due to~\eqref{eqn: FlowClose}.
\end{proof}

Using Lemma \ref{lemma: X_before_after} we can write now an alternative characterization of the time $\tau^\eps$ that will be used to explicitly write out the correction $X^\eps _{\tau^\eps} - z^i$.

Parameterize the hypersurface $M$  as a graph of a $C^2$-function $F^i$ over $T_zM$, i.e., $y\mapsto z^i+y+F^i(y)\cdot \bar \lambda_i(z)$ gives a $C^2$-parametrization of a neighborhood of $z^i$ in $M$ by a neighborhood of $0$ in $T_zM$. Moreover, $DF(0)=0$ so that $|F(y)|=O(|y|^2)$, $y\to 0$. With this definition, it is clear that, for $w\in \R^d$ with $w-z$ small enough, $w\in M$ if and only if $\pi_b(w-z)=F(\pi_M(w-z))$. Moreover, Lemma~\ref{lemma: X_before_after} and a direct application of Lemma 8 from~\cite{SergioBakhtin2011} gives that for every $\omega\in (\beta/2, \beta)$,
\[
\lim_{ \eps \to 0} \Pp \parbar{ \bigl\{\tau^\eps = \tilde \tau^\eps \} \cap  \left \{  |\tau^\eps - T^i| \leq \eps^\omega \right \} } =1,
\]
where $ \tilde \tau^\eps=\inf \{  t\geq 0: \pi^i \left(X^\eps_t-z^i \right) = F\left( \pi_M^i \left(X^\eps_t-z^i \right)  \right)\}\bigr \}$. Hence, from now on, we fix $\omega\in (\beta/2,\beta)$ and condition on the intersection of the events  $\Omega^\eps_1=\bigl \{ \tilde \tau^\eps =\tau^\eps \bigr \}$ and $\Omega^\eps_2= \bigl \{  | \tau^\eps- T^i|\leq \eps^\omega \bigr \}$.

Let $\hat \tau^\eps = T^i- \tau^\eps$, so that $\tau^\eps = T^i- \hat \tau^\eps$, and (since we are conditioning in $\Omega^\eps_2$ ) $|\hat \tau^\eps | < \eps^\omega$. In case $\hat \tau^\eps >0$, the conditioning on $\Omega^\eps_1$, and~\eqref{eqn: X_before} imply that
\begin{align*}
-\hat \tau^\eps+\eps^\beta \pi^i \parbar{ \eta^\eps_{T^i-\hat \tau^\eps} (\ell) + \Upsilon^{ \eps, -,i }_{ \hat\tau^\eps }} = F \parbar { D^{\eps,+} },
\end{align*}
where $D^{\eps,+} = -\hat \tau^\eps \pi^i_M\bar\lambda_i (z) + \eps^\beta \pi^i_M\parbar{ \eta^\eps_{T^i-\hat \tau^\eps} (\ell) + \Upsilon^{ \eps, -,i}_{ \hat\tau^\eps } }$. Hence, since $2\omega > \beta$, it follows that
\begin{equation} \label{eqn:conv1}
\one\parbar{ \{ T^i> \tau_\eps\} \cap \Omega^\eps_1 \cap \Omega^\eps_2} \eps^{-\beta} \parbar{ ( \tau^\eps - T^i ) - \pi^i \eta^\eps_{ T^i - \hat\tau^\eps }  (\ell) }  \stackrel{\Pp}{\longrightarrow} 0,\quad\eps\to0.
\end{equation}
In the case $\hat \tau^\eps < 0$, the reasoning is completely analogous. Indeed, using~\eqref{eqn: X_after}, we can get that
\begin{equation} \label{eqn:conv2}
\one\parbar{ \{ T^i \leq \tau_\eps\} \cap \Omega^\eps_1 \cap \Omega^\eps_2} \eps^{-\beta} \parbar{ ( \tau^\eps - T^i ) - \pi^i \tilde \eta^\eps_{  - \hat\tau^\eps } (\ell)  }  \stackrel{\Pp}{\longrightarrow} 0,\quad\eps\to0.
\end{equation}
By  adding up~\eqref{eqn:conv1},and~\eqref{eqn:conv2}, and using~\eqref{eqn: tildeetaconv} in~\eqref{eqn:conv1} it follows that
\[
\eps^{-\beta} \parbar { \tau^\eps - T^i   } \to \pi^i \eta^i_T( \ell), \text{ in distribution, as }\eps\to0.
\]

These computations give us the convergence of the time component. Once we have the time component, the spatial component and the joint convergence follows as in [1]. This completes the proof of Theorem \ref{thm: Main}.

\section{First Order Langevin Equation in a  Rough Potential}\label{S:ConditionalExitLaw}

In this section we apply Theorems \ref{T:CLT2} and \ref{thm: Main} to a small noise diffusion process in a rough environment in Regime $1$, i.e. we  assume
 that $\delta$ goes to zero faster than $\epsilon$ does. We change the notation to include the dependence on $\delta$.

To be precise, we consider the first order Langevin equation in a rough potential, defined as
\begin{equation}
d\x _{t}=\left[  -\frac{\epsilon}{\delta} \nabla Q\left(
\frac{\x_{t}}{\delta}\right)  - \nabla V\left(  \x_{t}\right)  \right]  dt+\sqrt{2\eps D}dW_{t},\hspace{0.2cm}%
X^{\epsilon}_{0}=x^{\epsilon}_{0}. \label{Eq:LangevinEquation2}%
\end{equation}
where $x^{\epsilon}_{0}\rightarrow x_{0}$ as $\epsilon\downarrow 0$. Note that the rough potential $V^{\epsilon,\delta}\left(x,\frac{x}{\delta}\right)=\epsilon Q(x/\delta)+V(x)$ is composed by a large scale smooth part
, $V(x)$, and by a smaller scale fast oscillating part, $\epsilon Q(x/\delta)$. We assume that $V$ is a $C^3(\R^{d})$ function while $Q$ is $C^2(\R^{d})$ with period $\rho$.

By~\eqref{Eq:LLN} we know that $\x$ converges in probability, uniformly in $t\in[0,T]$, as $\epsilon/\delta\uparrow\infty$, to $\bar{X}$ where $\bar X$ is the solution to the ODE
\[
\dot{\bar{X}}_{t}=\bar{\lambda}\left(\bar{X}_{t}\right),\quad \bar X_0 = x_0,
\]
driven by the vector field $\bar{\lambda}$ defined in Definition~\ref{Def:ThreePossibleFunctions}. In this case is easy to see that the invariant measure
$\mu(dy)$ of $\mathcal{L}=-\nabla Q \partial_x + D \partial_x^2$ is given by the Gibbs distribution $\mu(dy)=K^{-1}e^{-\frac{Q(y)}{D} }dy$, where
\[
K=\int_{\mathcal{Y}}e^{-\frac{Q(y)}{D}}dy.
\]
In dimension $d=1$, after some algebra, we get that
  $$\bar \lambda (x)=-\frac{\rho^{2}}{K\hat{K}} V'(x),$$ where $\hat{K}=\int
_{\mathbb{T}}e^{\frac{Q(y)}{D}}dy.$

Let us first see how the central limit type of Theorem \ref{T:CLT2} translates in this special case of interest. In many problems of interest, one is
interested in understanding the behavior of the process starting within the neighborhood of a stable point of $V(x)$, assume that such a point is $x=x_{0}$. In Figure \ref{F:Figure1a}, we see a simple example of such a potential function.
To account for this fact
we assume that $x^{\epsilon}_{0}=x_{0}+\epsilon^{a_{2}/2}\xi^{\epsilon}$, where the random variable $\xi^{\epsilon}\rightarrow \xi^{0}$ in distribution as $\epsilon\downarrow 0$.

The following proposition states the central limit theorem in this particular case. We point out the presence of the additional drift term $\bar{J}_{1}(\bar{X}_{t})$.
\begin{proposition}
Consider the solution to the SDE (\ref{Eq:LangevinEquation2}) in $t\in[0,T]$ where $\epsilon/\delta\uparrow\infty$. Under the setup of Theorem \ref{T:CLT2} we have that the process $\eta^{\epsilon}_{t}(\ell)=\frac{X^{\epsilon}_{t}-\bar{X}_{t}}{\beta^{\epsilon}(\ell)}$ converges in distribution in the space of
continuous functions in $\mathcal{C}\left([0,T];\mathbb{R}^{d}\right)$ to the process $\bar{\eta}_{t}$ which is as follows.
\begin{enumerate}
 \item{If $\ell=0$, then $\bar{\eta}$ satisfies the ODE
 \[
  d\bar{\eta}_{t}=\left[D\bar{\lambda}(\bar{X}_{t})\bar{\eta}_{t}+\bar{J}_{1}(\bar{X}_{t})\right]dt, \quad \bar{\eta}_{0}=0
 \]
 .}
\item{If $\ell\in(0,\infty]$, then $\bar{\eta}$ is solution to the Ornstein-Uhlenbeck process
\[
d\bar{\eta}_{t}=\left[D\bar{\lambda}(\bar{X}_{t})\bar{\eta}_{t}+\ell^{-1}\bar{J}_{1}(\bar{X}_{t})\right]dt+\one\{m=1/2\}\bar{q}^{1/2}_1 dW_{t},\quad \bar{\eta}_{0}=\xi^{0}\one\{m=a_{2}/2\}.
\]
In dimension $d=1$, we have that $\bar \lambda (x)=-\frac{\rho^{2}}{K\hat{K}} V'(x)$ and $\bar{q}_1=-\frac{\rho^{2}2D}{K\hat{K}}$ and
\[
\bar{J}_{1}(x)=-\frac{\rho}{K\hat{K}D}|V'(x)|^{2}\int_{\mathbb{T}}\left[\left(1-\frac{\rho}{\hat{K}}e^{\frac{Q(y)}{D}}\right)\int_{\mathbb{T}\bigcap\{z\leq y\}}\left(1-\frac{\rho}{K}e^{-\frac{Q(z)}{D}}\right)dz\right]dy.
\]}
\end{enumerate}
\end{proposition}

The proof is a straightforward application of Theorem \ref{T:CLT2} and thus omitted. Notice that unless $Q=0$ or $\ell=\infty$, the term  $\bar{J}_{1}(x)\neq 0$ has non zero contribution in the limiting fluctuation process.

Next, we study the related conditional exit law, the result is in Theorem \ref{T:ConditionalExitTime}. From now on we assume that the initial point is $x^{\epsilon}_{0}=x_{0}$ and restrict the
analysis to dimension $d=1$ and to taking first $\delta\downarrow 0$ with $\epsilon$ fixed and then taking $\epsilon\downarrow 0$. Essentially, this corresponds to the case $\ell=\infty$.

Let us assume for concreteness that $V(x)$ is strictly convex, has unique minimum at $x=z_{0}$ such that $V(z_{0})=V'(z_{0})=0$, $V(x)>0$
for $x\neq z_{0}$ and $V'(x)\neq 0$ for $x\neq z_{0}$. Without loss of generality we assume that $z_{0}=0$. Consider an interval
$I=[x_{-},x_{+}]$ containing $x_{0}$ and assume that $0<x_{-}$. In Figure \ref{F:Figure1a}, we see a simple example of such a potential function.

\begin{figure}
[ptb]
\begin{center}
\includegraphics[height=2.8253in,width=3.1782in]{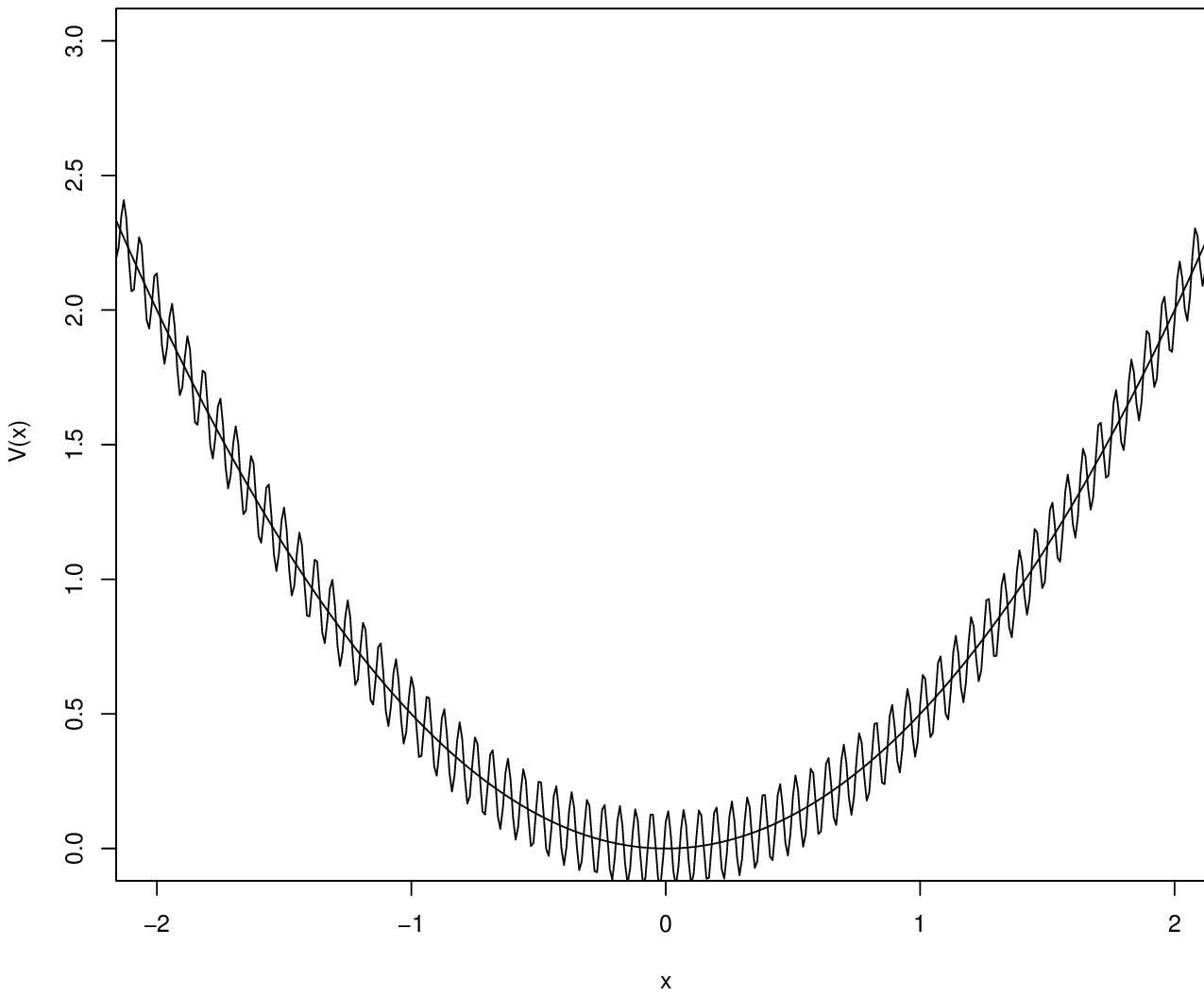}
\caption{$Q(x/\delta)=\cos(x/\delta)+\sin(x/\delta)$, $V(x)=\frac{1}{2}x^{2}$ with $\epsilon=0.1,\delta=0.01$.}%
\label{F:Figure1a}
\end{center}
\end{figure}

Let us define the exit time
\[
\tau^{\epsilon, \delta}=\inf\left\{  t\geq0:\x(t)\notin(x^{-},x^{+})\right\},
\]
and consider the event $B^{\epsilon, \delta}=\left\{ \x_{\tau^{\epsilon,\delta}}=x_{-}\right\}$. From the assumptions of $V$, it follows that $\lim_{\eps \to 0} \Pp \left\{ B^{\epsilon,\delta} \right\}=0.$ Large deviations for such and more general processes of similar structure have been studied in~\cite{DupuisSpiliopoulos}. Our goal in this section is to study the behavior of the exit problem from $I$ for this process conditioned on the rare event $B^{\epsilon,\delta}$ when $\delta$ goes to zero much faster than $\epsilon$ and provide a limit theorem using Theorem~\ref{thm: Main}. With some abuse of notation, we shall denote this process as $\x |_{ B^{\epsilon,\delta} }$.
In Remark \ref{R:RandomEnvironment} we discuss the case when  $Q$ is a stationary and ergodic random field.

We investigate how the fast oscillations of the small perturbation function $Q(y)$ affect the conditional exit law.
We first derive the process to which $\x|_{B^{\epsilon,\delta} }$ converges to by first taking $\delta\downarrow 0$ with $\epsilon$ fixed and then taking $\epsilon\downarrow 0$. We do this using the Feller characterization of one-dimensional diffusion processes given in~\cite{Feller} and the related weak convergence results of \cite{FreidlinWentzell1994}. These results are recalled for the convenience of the reader in Appendix~\ref{S:Prelim}.

In order to formulate our results we need to introduce some more notation. Let us define the 1-dimensional torus that the fast motion takes place as $\mathbb{T}^{1}$. We shall use the notation,
\begin{eqnarray}
 \left< g \right>=\frac{1}{\rho}\int_{\mathbb{T}^{1}}g(z)dz
\end{eqnarray}
for the mean value of a periodic function $g$ with period $\rho$.

\begin{theorem}\label{T:ConditionalProcess}
Let $S>0$ be given. Given the assumptions made on the functions $Q(y)$ and $V(x)$ above, we have that conditioned on $B^{\epsilon,\delta}$, the process $\x$ converges weakly, as $\delta\downarrow 0$ with $\epsilon$  fixed, in the space of continuous function $\mathcal{C}\left([0,S];\mathbb{R}\right)$ to a process which at least up to the time that it exits the interval $I=[x_{-},x_{+}]$, satisfies

\begin{equation}
d\hat{X}^{\epsilon}_{t}=\left[\frac{V'\left(\hat{X}^{\epsilon}_{t}\right)}{\left<e^{-\frac{Q}{D}}\right>\left<e^{\frac{Q}{D}}\right>}+\Psi^{\epsilon}(\hat{X}^{\epsilon}_{t})\right]dt
+\sqrt{\epsilon}\frac{\sqrt{2D}}{\sqrt{\left<e^{-\frac{Q}{D}}\right>\left<e^{\frac{Q}{D}}\right>}}dW_{t}, \quad \hat{X}_{0}=x_{0}
\end{equation}
where the function $\Psi^{\epsilon}(x)=o(\epsilon)$ as $\epsilon\downarrow 0$ uniformly in $x$.
\end{theorem}

The proof of this theorem is deferred to the end of this section. With this result at hand, Theorem \ref{thm: Main}, implies the following result for the limiting distribution of the conditional exit time $\tau^{\epsilon,\delta}$.

\begin{theorem}\label{T:ConditionalExitTime}
Let
\[
T(x_0) = \left<e^{-\frac{Q}{D}}\right>\left<e^{\frac{Q}{D}}\right> \int_{x_0}^{x_-} \frac{dy}{V^\prime (y) } < \infty.
\]
Given the assumptions made on the functions $Q(y)$ and $V(x)$ above, we have that conditioned on $B^{\epsilon,\delta}$, the distribution of $\frac{1}{\sqrt{\epsilon}}(\tau^{\epsilon,\delta}-T(x_{0}))$ converges weakly, as first $\delta\downarrow 0$ and then $\epsilon\downarrow0$ to a Gaussian random variablee with mean zero and variance given by
\begin{equation}
2D\left|\left<e^{-\frac{Q}{D}}\right>\left<e^{\frac{Q}{D}}\right>\right|^{2}\int_{x_{0}}^{x_{+}}\frac{1}{(V'(z))^{3}}dz.\label{Eq:LimitingVariance}
\end{equation}
\end{theorem}

\begin{proof}
We only give a sketch of the main arguments, since based on Theorems \ref{T:ConditionalProcess} and \ref{thm: Main}, the proof follows along the lines of Theorem 2 in \cite{SergioBakhtin2011}.  For notational convenience let us define
\[
\bar{\sigma}=\frac{\sqrt{2D}}{\sqrt{\left<e^{-\frac{Q}{D}}\right>\left<e^{\frac{Q}{D}}\right>}}
\]
and recall that
\[
\bar{\lambda}(x)=-\frac{V'\left(\hat{X}^{\epsilon}_{t}\right)}{\left<e^{-\frac{Q}{D}}\right>\left<e^{\frac{Q}{D}}\right>}.
\]

Let us define $\hat{X}_{t}$ to be the solution to the ODE
\[
\frac{d}{dt}\hat{X}_{t}=-\bar{\lambda}\left(\hat{X}_{t}\right), \quad \hat{X}_{0}=x_{0}
\]
and let $\hat{\Phi}_{x_{0}}(t)$ solving the ODE
\[
\frac{d}{dt}\hat{\Phi}_{x_{0}}(t)=-\bar{\lambda}^{'}\left(\hat{X}_{t}\right)\hat{\Phi}_{x_{0}}(t), \quad \hat{\Phi}_{x_{0}}(0)=1
\]

By Theorem \ref{T:ConditionalProcess} and applying Theorem \ref{thm: Main} for $i=1$, we obtain that
\[
\frac{1}{\sqrt{\epsilon}}(\tau^{\epsilon,\delta}-T(x_{0}))\rightarrow -\frac{1}{\bar{\lambda}(x_{+})}\hat{\Phi}_{x_{0}}\left(T(x_{0})\right)\int_{0}^{T(x_{0})}\hat{\Phi}^{-1}_{x_{0}}\left(s\right)\bar{\sigma}d\hat{W}_{s},
\]
weakly, as first $\delta\downarrow 0$ and then $\epsilon\downarrow 0$. Thus, the limit is a centered Gaussian random variable. The rest of the proof amounts to proving that the variance of this Gaussian random variable reduces to (\ref{Eq:LimitingVariance}); this is done in a similar situation in the proof of Theorem 2 of  \cite{SergioBakhtin2011} and thus omitted. This concludes the sketch of the proof of the theorem.
\end{proof}

The results hold in the case of a random environment. In particular we have the following remark.
\begin{remark}\label{R:RandomEnvironment}
Even though we have stated  Theorems \ref{T:ConditionalProcess} and \ref{T:ConditionalExitTime} only for a periodic function $Q(y)$, the proof of Theorem \ref{T:ConditionalProcess} below immediately shows that the statements are true also when $Q(y)$ is  a stationary, ergodic
random field defined on some probability space $(\Psi,\mathcal{G},\nu)$. For
every $\omega\in\Psi$, $Q(y,\omega)$ is $\mathcal{C}^{2}(\mathbb{R})$ in
$y$ with bounded and Lipschitz continuous derivatives up to order 2. In particular, in this case we have
\begin{equation*}
\left<e^{-\frac{Q}{D}}\right>=E^{\nu}\left[e^{-\frac{Q(y)}{D}}\right],\quad \left<e^{\frac{Q}{D}}\right>=E^{\nu}\left[e^{\frac{Q(y)}{D}}\right]
\end{equation*}
where $E^{\nu}$ is expectation under the random environment.  In the case of Theorem \ref{T:ConditionalExitTime}, it seems plausible to prove that the convergence is weak in $\mathcal{C}\left([0,T];\mathbb{R}\right)$, in probability with respect to $\nu$.
\end{remark}

From Theorems \ref{T:ConditionalProcess} and \ref{T:ConditionalExitTime}, we can get some interesting conclusions on the effect of the small but fast oscillations $\epsilon Q(x/\delta)$ on the underlying potential $V(x)$. We have the following remark

\begin{remark}\label{R:ConclusionsRoughPotential}
Theorem \ref{T:ConditionalExitTime} gives a second order approximation for $\tau^{\epsilon,\delta}$ conditioned on $B^{\epsilon,\delta}$ when $\epsilon,\delta$ are small. We notice that compared to the process without any fast oscillations (i.e. take $Q(y)=0$), the standardized limiting conditional exit law has variance multiplied by the constant $\left|\left<e^{-\frac{Q}{D}}\right>\left<e^{\frac{Q}{D}}\right>\right|^{2}$. By H\"{o}lder's inequality, it is easy to see that
\[
\left<e^{-\frac{Q}{D}}\right>\left<e^{\frac{Q}{D}}\right>\geq 1
\]
Therefore, the limiting conditional variance has been enhanced by the factor $\left|\left<e^{-\frac{Q}{D}}\right>\left<e^{\frac{Q}{D}}\right>\right|^{2}$ due to the fast oscillations.
\end{remark}

We conclude this section with the proof of Theorem \ref{T:ConditionalProcess}.
\begin{proof}[Proof of Theorem \ref{T:ConditionalProcess}]
By Lemma 3 in \cite{SergioBakhtin2011} we know that  conditioned on $B^{\epsilon,\delta}$, $\x $ behaves, for each $\epsilon,\delta>0$, as a diffusion process with infinitesimal generator $\tilde{L}^{\epsilon,\delta}$ given by
\begin{equation*} \tilde{L}^{\epsilon,\delta}= b^{\epsilon,\delta}\left(x,\frac{x}{\delta}\right)\partial_x + \epsilon D \partial^2_x,
\end{equation*}
where
\begin{equation*}
b^{\epsilon,\delta}\left(x,\frac{x}{\delta}\right)= -\frac{\epsilon}{\delta}Q'\left(\frac{x}{\delta}\right)-V'\left(x\right)+2 \epsilon D \frac{h^{\epsilon,\delta}(x)}{\int_{0}^{x}h^{\epsilon,\delta}(y)dy}
\end{equation*}
and
\begin{equation*}
h^{\epsilon,\delta}(x)=\exp \left \{\frac{1}{\epsilon D }\int_{0}^{x}\left[\frac{\epsilon}{\delta}Q'\left(\frac{y}{\delta}\right)+V'\left(y\right)\right]dy \right \}.
\end{equation*}

Then, as it is ease to see, the operator $\tilde{L}^{\epsilon,\delta}$ can be equivalently written in the $D_{v^{\epsilon,\delta}}D_{u^{\epsilon,\delta}}$
characterization of Feller \cite{Feller} (see Appendix~\ref{S:Prelim} for some related results from the literature). In this case, the corresponding $u^{\epsilon,\delta}(x)$ and $v^{\epsilon,\delta}(x)$ functions are defined as
\begin{align}
u^{\epsilon,\delta}(x)&=\int_{0}^{x} \exp \left\{-\frac{1}{\epsilon D}\int_{0}^{y} b^{\epsilon,\delta}(z,\frac{z}{\delta}) dz \right\} dy, \text{ and }\nonumber\\
v^{\epsilon,\delta}(x)&=\int_{0}^{x}\frac{1}{\epsilon D} \exp \left \{\frac{1}{\epsilon D}\int_{0}^{y} b^{\epsilon,\delta}(z,\frac{z}{\delta})dz \right \}\nonumber
dy.
\end{align}

By \cite{FreidlinWentzell1994}, we know that if $u^{\epsilon}(x),v^{\epsilon}(x)$ are the limits of $u^{\epsilon,\delta}(x),v^{\epsilon,\delta}(x)$ as $\delta\downarrow 0$, then the process corresponding to the operator $D_{v^{\epsilon,\delta}}D_{u^{\epsilon,\delta}}$ will converge weakly in $\mathcal{C}\left([0,S];\mathbb{R}\right)$ to the process corresponding to the operator $D_{v^{\epsilon}}D_{u^{\epsilon}}$. Our task now is to investigate these limits.

Let us first investigate $u^{\epsilon,\delta}(x)$. To simplify notation, denote $\Phi(x)=-\frac{Q'(x)}{D}$, $\Psi(x)=-\frac{V'(x)}{D}$ and
 $\zeta(y)=\int_{0}^{y}\Phi(\rho)d\rho$. With this notation at hand, we have
\begin{align} \notag
u^{\epsilon,\delta}(x)&=\int_{0}^{x} \exp \left\{-\frac{1}{\epsilon D}\int_{0}^{y} b^{\epsilon,\delta}(z,\frac{z}{\delta}) dz \right\} dy \\ \notag
&=\int_{0}^{x} \exp \left\{-\zeta\left(\frac{y}{\delta}\right)-\frac{1}{\epsilon}\int^{y}_{0}\Psi\left(z\right)dz \right\}
\exp \left \{-2\int^{y}_{0} \frac{h^{\epsilon,\delta}(z)}{\int_{0}^{z}h^{\epsilon,\delta}(w)dw}dz \right\}dy\\ \notag
&= \int_{0}^{x} \exp \left\{-\zeta\left(\frac{y}{\delta}\right)-\frac{1}{\epsilon}\int^{y}_{0}\Psi\left(z\right)dz\right\} \\
& \qquad  \times \exp \left \{-2\int^{y}_{0} \frac{ \exp \left \{-\zeta\left(\frac{z}{\delta}\right)-\frac{1}{\epsilon}\int^{z}_{0}\Psi\left(\rho\right)d\rho\right\} }
{\int_{0}^{z}\exp \left\{ -\zeta\left(\frac{w}{\delta}\right)-\frac{1}{\epsilon}\int^{w}_{0}\Psi\left(\rho\right)d\rho \right\}dw}dz \right\}dy. \label{eqn: uepsdeltaLong}
\end{align}
By the mean value theorem we know that for a periodic function $g\in L^{a}(\mathcal{Y})$, for $a\geq 1$,
we have that $g\left(\frac{x}{\delta}\right)\rightharpoonup \left< g \right>$ in $L^{a}_{\textrm{loc}}(\mathbb{R})$ as $\delta\downarrow 0$. The convergence is in the weak sense in the spaces of functions for any arbitrary bounded interval in $\mathbb{R}$. Using this on~\eqref{eqn: uepsdeltaLong}, we have that as $\delta\downarrow 0$,  $u^{\epsilon,\delta}(x)\rightarrow u^{\epsilon}(x)$, where
\begin{align} \notag
u^{\epsilon}(x)&=\int_{0}^{x} \left<e^{-\zeta}\right> \exp \left \{-\frac{1}{\epsilon}\int^{y}_{0}\Psi\left(z\right)dz\right \}
\exp \left \{-2\int^{y}_{0} \frac{\left<e^{-\zeta}\right>e^{-\frac{1}{\epsilon}\int^{y}_{0}\Psi\left(z\right)dz}}
{\int_{0}^{z}\left<e^{-\zeta}\right>e^{-\frac{1}{\epsilon}\int^{w}_{0}\Psi\left(\rho\right)d\rho}dw}dz \right \}
 dy\\
&=\int_{0}^{x} \left<e^{-\zeta}\right> \exp \left \{-\frac{1}{\epsilon}\int^{y}_{0}\Psi\left(z\right)dz \right \}
\exp \left \{-2   \ln \left|e^{-\frac{1}{\epsilon}\int^{y}_{0}\Psi\left(z\right)dz} -1\right| \right \}
 dy.\label{eqn: ueps1}
\end{align}
Due to our assumptions, we have that  $\int^{y}_{0}\Psi\left(z\right)dz<0$ for $y\geq 0$. Thus, it is easy to see that as $\epsilon\downarrow 0$
\begin{equation*}
\ln\left(e^{-\frac{1}{\epsilon}\int^{y}_{0}\Psi\left(z\right)dz} -1 \right)=-\frac{1}{\epsilon}\int^{y}_{0}\Psi\left(z\right)dz+\frac{1}{2 \epsilon} \Psi^{\epsilon}(y)
\end{equation*}
where $\Psi^{\epsilon}(y)=o(\epsilon)$ uniformly in $y$. Hence, this and~\eqref{eqn: ueps1} implies that as  $\epsilon\downarrow 0$ we have
\begin{align*}
u^{\epsilon}(x)&=\left<e^{-\zeta}\right>\int_{0}^{x} \exp \left\{-\frac{1}{\epsilon}\int^{y}_{0}\Psi\left(z\right)dz \right\} \exp \left\{-2 \ln\left(e^{-\frac{1}{\epsilon}\int^{y}_{0}\Psi\left(z\right)dz} -1 \right) \right\}   dy\nonumber\\
&=\left<e^{-\zeta}\right>\int_{0}^{x} \exp \left \{-\frac{1}{\epsilon}\int^{y}_{0}\Psi\left(z\right)dz\right\}
\exp \left \{-2 \left(-\frac{1}{\epsilon}\int^{y}_{0}\Psi\left(z\right)dz +\frac{1}{2 \epsilon}\Psi^{\epsilon}(y)\right) \right\}   dy\nonumber\\
&=\left<e^{-\zeta}\right>\int_{0}^{x} \exp \left \{-\frac{1}{\epsilon}\left(-\int^{y}_{0}\Psi\left(z\right)dz + \Psi^{\epsilon}(y)\right) \right\}   dy\nonumber
\end{align*}

On the other hand, in exactly the same way we get that $v^{\epsilon,\delta}(x)\rightarrow v^{\epsilon}(x)$, $\delta \downarrow 0$, where
\begin{equation}
v^{\epsilon}(x)
=\left<\frac{e^{\zeta}}{\epsilon D}\right>\int_{0}^{x} \exp \left \{\frac{1}{\epsilon}\left(-\int^{y}_{0}\Psi\left(z\right)dz +\Psi^{\epsilon}(y)\right)\right\}   dy\nonumber\\
\end{equation}
with the same $\Psi^{\epsilon}(y)=o(\epsilon)$ as $\epsilon\downarrow 0$.

The limiting $u^{\epsilon}(x)$ and $v^{\epsilon}(x)$ correspond to a process characterized by the generator (see Remark~\ref{R:ReductionToSimpleCase})
\begin{equation*}
 \hat{L}^{\epsilon}=\left(\hat{b}(x)+\Psi^{\epsilon}(x)\right)\partial_x+\epsilon\hat{\alpha} \partial^2_x
\end{equation*}
where
\begin{eqnarray}
 \hat{\alpha}&=&\frac{2D}{\left<e^{-\frac{Q}{D}}\right>\left<e^{\frac{Q}{D}}\right>}\nonumber\\
\hat{b}(x)&=&\frac{V'\left(x\right)}{\left<e^{-\frac{Q}{D}}\right>\left<e^{\frac{Q}{D}}\right>}\nonumber
\end{eqnarray}
which concludes the proof of the theorem.
\end{proof}

\appendix
\section{Feller's Characterization of One Dimensional Diffusions}\label{S:Prelim}

Consider a stochastic process in one-dimensional characterized by its generator
\begin{equation}
 Lf(x)=\frac{1}{2}\alpha(x)\frac{d^{2}f}{dx^{2}}+b(x)\frac{df}{dx}\label{SmoothOperator}
\end{equation}
with smooth enough coefficients $a(x)>0$ and $b(x)$.
For the convenience of the reader, we briefly recall the Feller
characterization of all one-dimensional Markov processes, that are
continuous with probability one (for more details see \cite{Feller};
also \cite{M1}). All one-dimensional strong Markov processes that
are continuous with probability one, can be characterized (under
some minimal regularity conditions) by a generalized second order
differential operator $D_{v}D_{u}f$ with respect to two increasing
functions $u(x)$ and $v(x)$ and its domain of definition. In particular,  $u(x)$ is continuous and $v(x)$ is right
continuous. In addition, $D_{u}$, $D_{v}$ are differentiation
operators with respect to $u(x)$ and $v(x)$ respectively, which are
defined as follows:

$D_{u}f(x)$ exists if $D_{u}^{+}f(x)=D_{u}^{-}f(x)$, where the left
derivative of $f$ with respect to $u$ is defined as follows:
\begin{displaymath}
D_{u}^{-}f(x)=\lim_{h\downarrow 0}\frac{f(x-h)-f(x)}{u(x-h)-u(x)}
\hspace{0.2cm} \textrm{ provided the limit exists.}
\end{displaymath}
The right derivative $D_{u}^{+}f(x)$ is defined similarly. If $v$ is
discontinuous at $y$ then
\begin{displaymath}
 D_{v}f(y)=\lim_{h\downarrow 0}\frac{f(y+h)-f(y-h)}{v(y+h)-v(y-h)}.
\end{displaymath}
\begin{rem}\label{R:ReductionToSimpleCase}
For example, it is easy to see that the operator $L$  in
(\ref{SmoothOperator}) can be written as a $D_{v}D_{u}$ operator
with $u$ and $v$ as follows:
\begin{equation}
u(x)=\int_{0}^{x} e^{-\int_{0}^{y}\frac{2b(z)}{a(z)}dz} dy
\quad\textrm{ and }\quad
v(x)=\int_{0}^{x}\frac{2}{a(y)} e^{\int_{0}^{y}\frac{2b(z)}{a(z)}dz}
dy. \label{UandVfunctions}
\end{equation}
The representation of $u(x)$ and $v(x)$ in (\ref{UandVfunctions}) is
unique up to multiplicative and additive constants. One can
multiply one of these functions by some constant and divide the
other function by the same constant or add a constant to either of
them.

\end{rem}
Another useful result in this direction is that of \cite{FreidlinWentzell1994}, where it is proven that if we have a sequence of operators $\{D_{v^{\epsilon}}D_{u^{\epsilon}}, \epsilon>0\}$ uniquely characterizing a sequence of
Markov processes $\{X^{\epsilon},\epsilon>0\}$ such that $u^{\epsilon}(x)\rightarrow u(x)$ and  $v^{\epsilon}(x)\rightarrow v(x)$ as $\epsilon\downarrow 0$, such that
$D_{v}D_{u}$ corresponds to a strongly continuous homogeneous Markov process $X$, then
$X^{\epsilon}_{\cdot}\rightarrow X_{\cdot}$ as $\epsilon\downarrow 0$
 in distribution in $\mathcal{C}([0,T];\mathbb{R})$ for every $T>0$.

\section{Acknowledgements}
K.S. was partially supported by the National Science Foundation (DMS 1312124).



\end{document}